\documentclass{article}
\usepackage{graphicx} 

\usepackage[swedish,english]{babel}
\usepackage[utf8]{inputenc}
\usepackage[T1]{fontenc}
\usepackage{amssymb}
\usepackage{multirow}

\usepackage{lipsum}

\usepackage[explicit]{titlesec}
\usepackage{amsmath}
\usepackage{amsthm}
\usepackage{bbm}
\usepackage{graphicx}
\usepackage{subcaption}
\usepackage[colorinlistoftodos]{todonotes}
\usepackage{tablefootnote}
\usepackage{hyperref}
\usepackage{float}
\usepackage{algpseudocode}
\usepackage{algorithm}
\usepackage{placeins}
\usepackage{changes}

\newtheorem{theorem}{Theorem}[]
\newtheorem{corollary}{Corollary}[section]
\newtheorem{lemma}[]{Lemma}
\newtheorem{proposition}{Proposition}[section]

\theoremstyle{definition}

\DeclareMathOperator{\diag}{diag}

\DeclareMathOperator{\tr}{tr}

\title{Novel Limited Memory Quasi-Newton Methods Based On Optimal Matrix Approximation}
\author{Erik Berglund and Mikael Johansson}
\date{March 2024}

\begin{document}

\maketitle

\begin{abstract}
    Update formulas for the Hessian approximations in quasi-Newton methods such as BFGS can be derived as analytical solutions to certain nearest-matrix problems. In this article, we propose a similar idea for deriving new limited memory versions of quasi-Newton methods. Most limited memory quasi-Newton methods make use of Hessian approximations that can be written as a scaled identity matrix plus a symmetric matrix with limited rank. We derive a way of finding the nearest matrix of this type to an arbitrary symmetric matrix, in either the Frobenius norm, the induced $l^2$ norm, or a dissimilarity measure for positive definite matrices in terms of trace and determinant. In doing so, we lay down a framework for more general matrix optimization problems with unitarily invariant matrix norms and arbitrary constraints on the set of eigenvalues. We then propose a trust region method in which the Hessian approximation, after having been updated by a Broyden class formula and used to solve a trust-region problem, is replaced by one of its closest limited memory approximations. We propose to store the Hessian approximation in terms of its eigenvectors and eigenvalues in a way that completely defines its eigenvalue decomposition, as this simplifies both the solution of the trust region subproblem and the nearest limited memory matrix problem. Our method is compared to a reference trust region method with the usual limited memory BFGS updates, and is shown to require fewer iterations and the storage of fewer vectors for a variety of test problems. 
\end{abstract}

\section{Introduction}

With the ever-increasing size of modern datasets, large-scale optimization remains an area with relevant applications. Optimization plays a foundational role in various machine learning approaches as they aim to minimize the discrepancy between model predictions and observed data \cite{SurveyofOptimizationML}. By using the empirical risk minimization principle \cite{ERM}, one often ends up needing to solve an unconstrained, nonlinear optimization problem. A particular challenge arises when the objective function depends on large amounts of data. In this case, function and gradient evaluations are expensive. In the classical parallelization strategy where gradient and function evaluations are distributed on multiple worker machines, the master node performing the iterate update tends to be idle waiting for the workers to return~\cite{AAF+:20}. This motivates the development of methods that can make as much use of previous function and gradient evaluations as possible under memory constraints, exploiting otherwise idle times of the master to perform increased computational work and reduce the total number of iterations needed to reach a target accuracy.

Limited memory quasi-Newton algorithms have seen much success for nonlinear optimization problems and implemented in various commercial software packages \cite{Numerical_optimization}. While 
limited memory quasi-Newton methods have traditionally been used for solving problems in a deterministic setting,  \cite{MultiBFGS} demonstrated that they can work in a stochastic setting with overlap between sampled batches. Ways to adapt limited memory quasi-Newton methods to stochastic problems is an active area of research \cite{SurveyML}, although it is not within the scope of this paper.

When considering the trade-off between computations per iteration, memory requirements, and the number of iterations, limited memory quasi-Newton methods can be seen as a hybrid of conjugate gradient methods and ordinary quasi-Newton methods. Shanno showed that a particular conjugate gradient method can be seen as a memoryless BFGS method \cite{Shanno_CG}. Requiring slightly more computational effort than conjugate gradient, while still being viable for large-scale problems, limited-memory quasi-Newton methods may strike a better trade-off between computations per iteration and the number of iterations to converge for some problems \cite{NumericalStudy}. 

Many quasi-Newton methods have been shown to obey a minimum-change principle when updating their Hessian approximations: They choose the closest matrix to the previous matrix in some distance measure, that satisfies the secant equation \cite{Papakonstantinou2009HistoricalDO}. This minimum change principle was proposed by Greenstadt \cite{Greenstadt1970VariationsOV}, and was used by Goldfarb \cite{Goldfarb_VM} to derive BFGS. One can motivate such a principle in two ways. Firstly, it discourages the cumulative growth of the Hessian approximations, making the method stable. Secondly, it encourages retaining as much information encoded in the previous matrix as possible. The typical limited memory variants of quasi-Newton methods, however, depart from this principle: They build Hessian approximations from stored iterates and gradients from the last $m$ iterations, discarding information from all previous iterations. While doing so limits memory requirements, it also limits how many iterations that a particular gradient and iterate can be taken into account. Thus, there is room for alternative approaches, that aim to limit memory while retaining information from all older iterations.

An important thing to note about limited memory quasi-Newton methods is the particular structure of their Hessian approximations. Starting with an identity matrix multiplied by a scalar and performing a limited number of quasi-Newton updates by methods in the Broyden class results in the sum of the initial scaled identity matrix, and a symmetric matrix of limited rank. This matrix structure is important in other contexts as well. Zha and Zhang \cite{Lowrankplusshift} denoted it as "low-rank plus shift", and used it to solve matrix approximation problems in latent semantic indexing. Explicit representations of Hessian approximations on this form were first derived for a few methods by Byrd, Nocedal and Schnabel \cite{Byrd_representations}. Later, different compact representations were derived for the full Broyden class by DeGuchy, Erway and Marcia \cite{Compact_representation_full_Broyden} and by Ek and Forsgren \cite{LM_QN_quadratic}. Such particular representations of Hessian approximations can be vital for the development of efficient quasi-Newton methods for large-scale problems. As an example of applications, Burdakov, Gong, Yuan and Zikrin \cite{Combining_TR_and_LQN} proposed several ways to efficiently combine limited memory quasi-Newton updates with a trust region method. By making use of the eigenvalue decomposition of the Hessian approximation, they showed that it is possible to solve the trust region subproblem in a much smaller subspace, spanned by the eigenvalues of the Hessian approximation. This eigenvalue decomposition could be obtained with low computational effort thanks to making use of the low-rank plus shift structure of the Hessian approximation.  

The usefulness of a low-rank plus shift structure is apparent when deriving alternative limited memory quasi-Newton methods. For example, the limited memory variable metric methods of Vl\v{c}ek and Luk\v{s}an \cite{Limited_VM_methods} are based on a factorization of a symmetric rank $m$ matrix as a $n \times m$ matrix multiplied by its transpose, where $n$ is the number of variables of the problem. They derive formulas for updating the rectangular factors of the low-rank matrices in the Hessian approximations, such that the distance between the old and the updated rectangular factor in a weighted Frobenius norm is minimized, under the constraint that the secant equation is satisfied. In the context of solving general nonlinear systems of equations, quasi-Newton methods can be used to generate an approximate Jacobian. Even in this setting, the matrix approximation, when initialized as a scaled identity matrix, will throughout the iterations be equal to the sum of this scaled identity matrix and a matrix with an upper limit to its rank proportional to the iteration number, although this matrix is in general not symmetric. In a paper from 2003, van de Rotten and Lunel \cite{VDR} use singular value decomposition of the low-rank matrix, setting the lowest nonzero singular value to 0 to keep the rank below a fixed value. The resulting limited memory Broyden's method is then used to solve large nonlinear systems of equations arising in simulations of a chemical reactor. Both the aforementioned methods focus on the low-rank matrix in their matrix distance measures, leaving the scaling factor to be determined by other means.

Another approach to limited memory quasi-Newton methods focuses on being able to reproduce the iterations of its full-memory counterpart when that is possible. The aggregated BFGS method by Berahas, Curtis and Zhou \cite{AGG_BFGS} initially stores pairs of differences in gradients and variables, but after reaching the limit on the number of such pairs that can be stored, modifies these pairs along the iterations to aggregate curvature information. This enables the algorithm to store the same matrix as the one obtained by the full memory BFGS method when the differences between consecutive iterates lie in a limited subspace. 

While following a least-change property and being able to reproduce the iterations of a full-memory quasi-Newton method are both desirable properties, there is to the best of our knowledge no previous limited-memory quasi-Newton method that has both of them. However, if one were to consider the structure of low-rank plus shift and could optimize both the low-rank matrix and the scaling factor of the identity matrix simultaneously, there is great potential to develop such methods. Although applying the solution of the matrix optimization problems involved could require more computational effort than conventional limited memory quasi-Newton methods per iteration, the trade-off could be worthwhile if it decreases the total number of iterations needed. This leads us to the research presented in this paper.  

In this paper, we present solutions to several matrix optimization problems involving matrices with the low-rank plus shift structure and show how they can be applied to provide a framework for limited memory adaptions of quasi-Newton methods in the Broyden class. In contrast to~\cite{Limited_VM_methods,VDR}, the algorithms in this framework are derived based on optimizing both the low-rank and the shift part of the limited-memory representation, as opposed to only the low-rank factors. 

In the first part of the paper, we start by considering a general class of problems, where a unitarily invariant norm of the difference of an $n \times n$ variable matrix and a fixed matrix of the same size shall be minimized, with an arbitrary constraint on the set of eigenvalues of the variable matrix. We show that for every problem in this class, there is an optimal solution where the variable matrix has the same eigenvectors as the constant matrix, implying that only the eigenvalues need to be used as optimization variables. This reduces the number of variables in the problem from $n^2$ to $n$. We also show a similar result for other matrix distance measures. We then characterize limited memory quasi-Newton Hessian approximations as symmetric matrices with an eigenvalue of multiplicity $n-m$ where $m\ll n$, showing that finding the closest matrix with this structure is a problem of the type considered. Finally, we derive explicit solutions to some matrix optimization problems where a constraint is that the solution must belong to this class of matrices. 

In the second part of this paper, we apply the results of the first to find a new way of adapting Broyden class methods to the limited memory setting. Special attention is given to organizing the computations in a numerically robust manner while ensuring a linear scaling of the computational complexity in the number of variables. Even though our method relies on eigendecomposition, we demonstrate that the per-iteration cost is of $O(nm^2+m^3)$, as by taking advantage of the structure of the matrices, one only needs to decompose an $m \times m$ symmetric matrix, where $m \ll n$. While our framework encompasses a multitude of algorithms that one could implement, we limit the scope of the implementation section by focusing on limited memory versions of BFGS. Inspired by Burdakov et al. \cite{Combining_TR_and_LQN}, we take advantage of the availability of eigenvectors and eigenvalues inherent in our limited memory adaptions of BFGS to create two trust-region algorithms. These algorithms are tested on logistic regression problems, randomly generated quadratic problems, and a suite of CUTEst problems, along with a trust-region method using the conventional L-BFGS Hessian approximation as well as the conventional line-search implementation. We find that the extra work needed for updating the eigendecomposition of the Hessian approximation in our two tested algorithms is compensated for by a significantly faster convergence for the logistic regression and randomly generated quadratic programming problems, and a competitive performance profile concerning function evaluations compared to the line-search implementation of L-BFGS on the CUTEst test problems.

\section{Nearest matrix problems}\label{sec:NLMMP}

In this section, we will analyze several matrix optimization problems. We will move from the general to the specific, starting by deriving structural results for minimization of some unitarily invariant matrix dissimilarity measures under eigenvalue constraints, and ending with efficient algorithms for specific nearest matrix problems tailored to the development of quasi-Newton methods.

\subsection{Eigenvalue-constrained optimization of unitarily invariant matrix dissimilarity measures}

There are many matrix dissimilarity measures, from the standard matrix norms to more exotic divergences, see~\cite{VeJ:15,LeW:18}. We will focus on unitarily invariant measures for which we will be able to derive efficient solutions to the relevant nearest matrix problems.
Using the lemmas given in the appendix (Section~\ref{sec:matrixdistance}), the following statements can be made:

\begin{theorem} \label{thm:eigenv}
Consider the matrix optimization problem
\begin{equation*}
\begin{aligned}
& \underset{X}{\text{minimize}}
& & \| X - A \| \\
& \text{subject to}
& & Eig(X) \in S \\
& & & X = X^T, X \in {\mathbb{R}}^{n \times n}
\end{aligned}
\end{equation*}
where $\| \cdot  \|$ is any unitarily invariant norm, $A$ is a real symmetric matrix and $Eig(X) \in S$ is an arbitrary constraint on the multiset of eigenvalues of $X$. If this problem has at least one optimal solution, then it has an optimal solution $\widehat X$ with the same eigenvectors as $A$.
\end{theorem}
\begin{proof}
Let $\bar X$ be an optimal solution of the problem. Then, Lemma~\ref{lemma:ABC} states that there exists a Hermitian matrix $\widehat X$ with the same eigenvectors as $A$ and the same eigenvalues as $\bar X$ such that $\|A-\widehat X\| \leq \|A -\bar X\|$. Since the eigenvectors of $\widehat X$ are the same as those of $A$ and $\widehat X$ is Hermitian, it must be real and symmetric. Since $\widehat X$ has the same eigenvalues as $\bar X$ and is real and symmetric, it is also a feasible solution, and by optimality of $\bar X$ it must also be optimal.
\end{proof}

\begin{theorem} \label{thm:eigenv_tr_logdet}
Consider the matrix optimization problem
\begin{equation*}
\begin{aligned}
& \underset{X}{\text{minimize}}
& & d(X,A) \\
& \text{subject to}
& & Eig(X) \in S \\
& & & X \succ 0, X \in {\mathbb{R}}^{n \times n}
\end{aligned}
\end{equation*}
where $A$ is a real positive definite matrix, $Eig(X) \in S$ is an arbitrary constraint on the set of eigenvalues of $X$, and the objective function is either the Stein loss
\begin{align*}
    d(X,A) &= \tr (X^{-1}A)-\log\det(X^{-1}A)-n
\intertext{the inverse Stein loss}
    d(X,A) &= \tr(XA^{-1})-\log\det(XA^{-1})-n
\intertext{or the symmetrized Stein loss}
    d(X,A) &= \tr (X^{-1}A)+\tr(XA^{-1})
\end{align*}
If this problem has at least one optimal solution, then it has an optimal solution $\widehat X$ with the same eigenvectors as $A$.
\end{theorem}

\begin{proof}
The proof proceeds analogously to the proof for Theorem \ref{thm:eigenv}, with the use of Lemma \ref{lemma:logdetineq} instead of Lemma \ref{lemma:ABC}.
\end{proof}

The utility of the above theorems lies in how they can be used to reduce the number of variables in certain optimization problems. If we have access to the eigenvectors of the matrix $A$ in the above problems, we can set the eigenvectors of $X$ equal to them and concern ourselves entirely with choosing the $n$ eigenvalues of $X$ in an optimal way. 

\subsection{Optimal limited memory approximations of matrices}

As mentioned in the introduction, limited memory quasi-Newton methods often lead to matrices with a low-rank plus shift structure. For the remainder of this paper, it will be useful to introduce some notation to refer to sets of such matrices. Hence, the set of $n \times n$ limited memory matrices with memory parameter $m$, denoted ${\mathcal L}_{m,n}$, will be defined as follows: 
\begin{align*} 
{\mathcal L}_{m,n} = \left\{ M \in \mathbb{R}^{n\times n} | M = \alpha I + \sum_{i=1}^m a_i u_i u_i^ T \right \}
\end{align*}

Here, $\alpha$ and $a_i$ are real-valued scalars, and $u_i$ are $n$-dimensional real column vectors. Hence, to define a general matrix with this structure, a total of $nm+m+1$ scalars need to be stored in memory. From Equation \eqref{eqn:limited_Broyden_rep}, it can be seen that the matrices generated by the usual limited memory adaptions of quasi-Newton methods of the Broyden class belong to this set of matrices when up to $\lfloor m/2 \rfloor$ differences in gradient and variable vectors are stored. As will be seen in the following sections, it will prove useful to characterize the matrices by the multiplicity of their eigenvalues. Note that there is no need to distinguish between algebraic and geometric multiplicity for real symmetric matrices. When this distinction is not needed, it will be referred to just as "multiplicity".

\begin{proposition}\label{thm:multiplicity}
${\mathcal L}_{m,n}$ is the set of $n\times n$ symmetric matrices with an eigenvalue of multiplicity at least $n-m$
\end{proposition}
\begin{proof}
Let $M = \alpha I + \sum_{i=1}^m a_i u_i u_i^T$. Then, the vectors in the subspace of vectors orthogonal to all vectors $u_i$ are eigenvectors of $M$ with eigenvalue $\alpha$. Since this space is at least $(n-m)$-dimensional, $\alpha$ is an eigenvalue of multiplicity at least $(n-m)$. Conversely, let $M$ be symmetric and have an eigenvalue $\lambda$ of multiplicity at least $(n-m)$. Let $\lambda_1,...,\lambda_n$ be all its eigenvalues, and enumerate them so that $\lambda_{m+1}=...=\lambda_n = \lambda$. Let $v_1,...v_n$ be a basis of corresponding orthonormal eigenvectors. Then, $M = M - \lambda I + \lambda I = \lambda I + \sum_{i=1}^n (\lambda_i - \lambda) v_iv_i^T = \lambda I + \sum_{i=1}^m (\lambda_i - \lambda)v_iv_i^T$, so $M\in {\mathcal L}_{m,n}$.
\end{proof}

By Proposition \ref{thm:multiplicity}, the constraint $X\in {\mathcal L}_{m,n}$ is of the type described in theorems~\ref{thm:eigenv} and~\ref{thm:eigenv_tr_logdet} above, and so when solving the problem of finding a nearest limited memory matrix in the matrix dissimilarity measures considered above, it is sufficient to consider the set of matrices with the same eigenvectors as the matrix from which to minimize the distance. We are now ready to use the previously derived results to solve some nearest limited memory matrix problems, under the constraint that the variable matrix $X \in \mathcal{L}_{m,n}$. 

\begin{theorem}\label{thm:lowest_dist_l2}
Let $A$ be an $n \times n$ real symmetric matrix with eigenvalues $\\ \lambda_1(A),...,\lambda_n(A)$ and eigenvectors $v_1(A),...,v_n(A)$. The optimal solutions to 
\begin{equation*}\label{eqn:l2_problem}
\begin{aligned}
& \underset{X}{\text{minimize}}
& & \| X - A \|_2 \\
& \text{subject to}
& & X \in {\mathcal L}_{m,n} \\
\end{aligned}
\end{equation*}
for which $v_k(X) = v_k(A), \ \forall k \in \{1,...,n\}$ are precisely those such that $\lambda_k(X) \in I_k(\delta) = [\lambda_k(A)-\delta,\lambda_k(A)+\delta] \ \forall k \in \{1,...,n\}$, where $\delta$ is the lowest positive real number such that at least $n-m$ of the intervals $I_k(\delta)$ intersect at the same point, and $n-m$ of the eigenvalues of $X$ lie at the intersection of those intervals.
\end{theorem}
\begin{proof}
When $X$ and $A$ have the same eigenvectors, $\|A-X\|_2 = \max_k |\lambda_k(A)-\lambda_k(X)|$. Let $\delta = \max_k |\lambda_k(A)-\lambda_k(X)|$. Then, $\lambda_k(X) \in I_k = [\lambda_k(A)-\delta,\lambda_k(A) + \delta]$. $\delta$ must be chosen so that at least $n-m$ eigenvalues of $X$ can be equal to each other for the constraint $X \in {\mathcal L}_{m,n}$ to be fulfilled. The optimal value of $\delta$ is thus the lowest one such that at least $n-m$ of the intervals $I_k$ intersect. For that $\delta$, they will intersect at one point, and then it is necessary and sufficient to set $n-m$ eigenvalues corresponding to those intersecting intervals to lie at the point of intersection.  
\end{proof}
Using this characterization of the optimal approximations of $A$ in $\mathcal{L}_{m,n}$ with respect to the $l^2$-norm, computing one is easy if an eigendecomposition of $A$ is available. This can be seen by noting that the first point where $n-m$ uniformly expanding intervals around $n-m$ eigenvalues of A will intersect is the arithmetic average of the highest and lowest eigenvalue in a sequence of consecutive eigenvalues of $A$. The task then becomes checking which of the $m+1$ sequences of $n-m$ consecutive eigenvalues has the lowest distance between the lowest and highest eigenvalues, and then setting them to their arithmetic average, along with the other eigenvalues between them, leaving all other properties of A unchanged, which can be done with complexity $O(n+m)$. 

As we will see, the $l^2$ has one property that sets it apart from all the other matrix dissimilarity measures that we will consider. While the optimization of all of them will involve setting consecutive eigenvalues to some form of mean, the mean only involves two eigenvalues in the sequence for the $l^2$-norm, while the means used for the other matrix dissimilarity measures are taken over all eigenvalues in the sequence. 

To characterize the solutions to the matrix optimization problems that follow, the following lemma will be needed:

\begin{lemma}\label{lemma:consecutive_numbers}
    Let $f(x,z)$ be a function from $\mathbb{R}^2$ to $\mathbb{R}$ such that if $(x - z)(y - z) \geq 0$ and $y > x$, then $\left(f(y,z)-f(x,z)\right) \left(x + y - 2z\right) > 0$. Let $S$ be a finite subset of $\mathbb{R}$, and $M = \{S,\nu_M\}$ be a multiset over $S$ such that $|M|=n$. Consider the problem 
    \begin{equation*}
    \begin{aligned}
    & \underset{X}{\text{minimize}}
    & & \sum_{x \in X} \sum_{j=1}^{\nu_X(x)} f(x,\bar x) \\
    & \text{subject to}
    & & X \subset M, \quad |X| = m  \\
    \end{aligned}
    \end{equation*}
    where $n > m$ and $\bar x$ is any real number. The optimal submultiset $X^*$ of $M$ must contain consecutive numbers, in the sense that $\forall s \in S$, if $s \in (\min(X^*),\max(X^*))$, $\nu_{X^{*}}(s) = \nu_M(s)$.
\end{lemma}
\begin{proof}
    Assume that $X^*$ is an optimal solution to the problem, and that $\exists s \in S$ such that $\min(X^*) < s < \max(X^*)$ and $\nu_{X^*}(s) < \nu_M(s)$. If $\bar x \leq s$, then $f(\max(X^*),\bar x) > f(s,\bar x)$, and if $\bar x > s$, $f(\min(X^*),\bar x) > f(s,\bar x)$. Either way, a submultiset of $M$ with a lower value of the objective is obtained by decreasing the occurrence of either $\max(X^*)$ or $\min(X^*)$ in $X^*$ by 1 while increasing the occurrence of $s$ by 1, contradicting the optimality of $X^*$.
\end{proof}

We can now solve the problem of finding the best limited-memory matrix approximation to a real symmetric matrix in the Frobenius norm,
\begin{theorem}\label{thm:lowest_dist_frobenius}
Let $A$ be an $n \times n$ real symmetric matrix with eigenvalues $\\ \lambda_1(A),...,\lambda_n(A)$ and eigenvectors $v_1(A),...,v_n(A)$. The optimal solutions to 
\begin{equation*}
\begin{aligned}
& \underset{X}{\text{minimize}}
& & \| X - A \|_F \\
& \text{subject to}
& & X \in {\mathcal L}_{m,n} \\
\end{aligned}
\end{equation*}
for which $v_k(X) = v_k(A), \ \forall k \in \{1,...,n\}$ are precisely those such that \\ $\lambda_i(X) = \lambda = \frac{\lambda_l(A)+\lambda_{l+1}(A)+...+\lambda_{l+n-m-1}(A)}{n-m}$ if $i \in \{l,...,l+n-m-1 \}$ and $\lambda_i(A)$ otherwise, and $l$ is an integer that minimizes $\sum_{i=l}^{l+n-m-1}(\lambda_{i}(A)-\lambda)^2$. 
\end{theorem}
\begin{proof}
When $X$ and $A$ have the same eigenvectors, $\|X-A\|_F^2 = \sum_{i=1}^n (\lambda_i(X)-\lambda_i(A))^2$. To satisfy the constraint $X\in {\mathcal L}_{m,n}$, $n-m$ of the eigenvalues of $X$ must be set equal to each other. To minimize the Frobenius norm, the $m$ eigenvalues of $X$ that are not bound by the constraint should be equal to the corresponding eigenvalues of $A$. Let $i_1,...,i_{n-m}$ be the indices of those eigenvectors of $A$ and $X$ for which it is optimal to set the corresponding eigenvalues of $X$ equal to each other, and let all those eigenvalues be equal to $\lambda$. Then, the optimal value of the squared Frobenius norm is $\|X-A\|_F^2 = \sum_{k=1}^{n-m} (\lambda_{i_k}-\lambda)^2$. Since $\lambda$ is the value that minimizes this sum, the derivative of the sum with respect to $\lambda$ must be 0, so $\sum_{k=1}^{n-m}2(\lambda - \lambda_{i_k}(A)) = 0 \Leftrightarrow \lambda = \frac{\sum_{k=1}^{n-m}\lambda_{i_k}(A)}{n-m}$. The choice of optimal indices $i_1,...,i_{n-m}$ can now be posed as the type of problem considered in Lemma \ref{lemma:consecutive_numbers}. As the eigenvalues of $A$ are indexed in nondecreasing order, the optimal choice of $i_1,...,i_{n-m}$ must be $n-m$ consecutive integers $l,l+1,,...,l+n-m-1$ according to the lemma.
\end{proof}
The computational cost of finding the nearest limited memory matrix in Frobenius norm using the results above is equal to the cost of performing an eigendecomposition, plus the cost of finding the optimal sequence of eigenvalues to take the average of. For the latter task, an algorithm must for each sequence of $n-m$ consecutive eigenvalues check the sum of squared differences between the eigenvalues in that sequence and their average. Then, all eigenvalues in the sequence with the lowest sum shall be set to their average. The complexity for this is $O(n+m)$ if the algorithm only computes the mean and sum of squared mean differences by explicitly considering all eigenvalues for the first sequence and then computes the sum of squared mean differences and the mean for the next sequence recursively from those of the previous sequence and the two eigenvalues that are not shared by the two sequences.

Although the Frobenius norm and the induced $l^2$ norm are unitarily invariant, they are not invariant under more general transformations of matrices. While quasi-Newton methods such as BFGS and DFP can be derived as solutions to matrix optimization problems involving the Frobenius norm, those problems include weight matrices to make them invariant under linear changes of variables. This also makes the update formulas invariant; If $\tilde x_k = A x_k$, then $\tilde s_k = A s_k$, $\tilde y_k = A^{-T} y_k$, and $\tilde B_k = A^{-T} B_k A^{-1}$, and the same update formula that could be used to compute $B_k$ with $s_k$ and $y_k$ also holds for $\tilde B_k$ with $\tilde s_k$ and $\tilde y_k$. However, using the same weighted norms to derive limited memory Hessian approximations can not be done using Theorem~\ref{thm:Unitary matrix norms}, as the weighted norms themselves are not unitarily invariant. However, the DFP and BFGS methods can also be characterized by the matrix dissimilarity measures considered in Theorem~\ref{thm:eigenv_tr_logdet}. In what follows, we solve matrix optimization problems involving them. A drawback of these matrix dissimilarity measures is that they are only applicable to positive definite matrices. 

\begin{theorem}\label{thm:lowest_dist_tr_logdet1}
    Let $A$ be an $n \times n$ real positive definite matrix with eigenvalues $\lambda_1(A),...,\lambda_n(A)$ and eigenvectors $v_1(A),...,v_n(A)$. The optimal solutions to
    \begin{equation*}
    \begin{aligned}
    & \underset{X}{\text{minimize}}
    & & \tr(XA^{-1}) - \log \det (XA^{-1}) \\
    & \text{subject to}
    & & X \in {\mathcal L}_{m,n}, X \succ 0 \\
    \end{aligned}
    \end{equation*}
    for which $v_k(X) = v_k(A) \ \forall \ k \in \{1,...,n\}$ are precisely those such that \\ $\lambda_i(X) = \lambda = \left(\frac{\lambda_l(A)^{-1}+\lambda_{l+1}(A)^{-1}+...+\lambda_{l+n-m-1}(A)^{-1}}{n-m}\right)^{-1}$ if $i \in \{l,...,l+n-m-1 \}$ and $\lambda_i(A)$ otherwise, and $l$ is an integer that minimizes $\sum_{i=l}^{l+n-m-1}\frac{\lambda}{\lambda_i(A)} - \log (\frac{\lambda}{\lambda_i(A)})$.
\end{theorem}
\begin{proof}
When $X$ and $A$ have the same eigenvectors,
\begin{equation*}
    \tr(XA^{-1}) - \log \det (XA^{-1}) = \sum_{i=1}^n \frac{\lambda_i(X)}{\lambda_i(A)} - \log \left(\frac{\lambda_i(X)}{\lambda_i(A)} \right).
\end{equation*}

As in the proof of Theorem~\ref{thm:lowest_dist_frobenius}, $n-m$ eigenvalues of $X$ must be set to one and the same value, and it is optimal to set the other eigenvalues equal to the corresponding eigenvalues of $A$, as the function $x - \log(x)$ has a unique global minimum in its domain for $x = 1$. Let $i_1,...,i_{n-m}$ be the indices of those eigenvectors of $A$ and $X$ for which it is optimal to set the corresponding eigenvalues of $X$ equal to each other, and let all those eigenvalues be equal to $\lambda$. Then, the part of $\tr(XA^{-1}) - \log\det(XA^{-1})$ that we need to optimize with respect to $\lambda$ is $\sum_{k=1}^{n-m} \frac{\lambda}{\lambda_{i_k}(A)} - \log \left(\frac{\lambda}{\lambda_{i_k}(A)} \right)$. Since $\lambda$ is the value that minimizes this sum, the derivative of the sum with respect to $\lambda$ must be 0, so $\sum_{k=1}^{n-m} \frac{1}{\lambda_{i_k}(A)} - \frac{1}{\lambda} = 0 \Leftrightarrow \lambda = \left(\frac{\lambda_{i_1}(A)^{-1}+\lambda_{i_2}(A)^{-1}+...+\lambda_{i_{n-m}}(A)^{-1}}{n-m}\right)^{-1}$. The choice of optimal indices $i_1,...,i_{n-m}$ can now be posed as the type of problem considered in Lemma \ref{lemma:consecutive_numbers}. As the eigenvalues of $A$ are indexed in non-decreasing order, the optimal choice of $i_1,...,i_{n-m}$ must be $n-m$ consecutive integers $l,l+1,,...,l+n-m-1$ according to the lemma.
\end{proof}

\begin{theorem}\label{thm:lowest_dist_tr_logdet2}
    Let $A$ be an $n \times n$ real positive definite matrix with eigenvalues $\lambda_1(A),...,\lambda_n(A)$ and eigenvectors $v_1(A),...,v_n(A)$. The optimal solutions to
    \begin{equation*}
    \begin{aligned}
    & \underset{X}{\text{minimize}}
    & & \tr(X^{-1}A) - \log \det (X^{-1}A) \\
    & \text{subject to}
    & & X \in {\mathcal L}_{m,n}, X \succ 0 \\
    \end{aligned}
    \end{equation*}
    for which $v_k(X) = v_k(A) \ \forall \ k \in \{1,...,n\}$ are precisely those such that \\ $\lambda_i(X) = \lambda = \frac{\lambda_l(A)+\lambda_{l+1}(A)+...+\lambda_{l+n-m-1}(A)}{n-m}$ if $i \in \{l,...,l+n-m-1 \}$ and $\lambda_i(A)$ otherwise, and $l$ is an integer that minimizes $\sum_{i=l}^{l+n-m-1}\frac{\lambda_i(A)}{\lambda} - \log (\frac{\lambda_i(A)}{\lambda})$.
\end{theorem}

\begin{proof}
When $X$ and $A$ have the same eigenvectors,
\begin{equation*}
    \tr(X^{-1}A) - \log \det (X^{-1}A) = \sum_{i=1}^n \frac{\lambda_i(A)}{\lambda_i(X)} - \log \left(\frac{\lambda_i(A)}{\lambda_i(X)} \right).
\end{equation*}

As in the proof of Theorem~\ref{thm:lowest_dist_frobenius}, $n-m$ eigenvalues of $X$ must be set to one and the same value, and it is optimal to set the other eigenvalues equal to the corresponding eigenvalues of $A$, as the function $x - \log(x)$ has a unique global minimum in its domain for $x = 1$. Let $i_1,...,i_{n-m}$ be the indices of those eigenvectors of $A$ and $X$ for which it is optimal to set the corresponding eigenvalues of $X$ equal to each other, and let all those eigenvalues be equal to $\lambda$. Then, the part of $\tr(X^{-1}A) - \log\det(X^{-1}A)$ that we need to optimize with respect to $\lambda$ is $\sum_{k=1}^{n-m} \frac{\lambda_{i_k}(A)}{\lambda} - \log \left(\frac{\lambda_{i_k}(A) }{\lambda} \right)$. Since $\lambda$ is the value that minimizes this sum, the derivative of the sum with respect to $\lambda$ must be 0, so $ \sum_{k=1}^{n-m} - \frac{\lambda_{i_k}(A)}{\lambda^2} +\frac{1}{\lambda}= 0 \Leftrightarrow \lambda = \frac{\lambda_{i_1}(A)+\lambda_{i_2}(A)+...+\lambda_{i_{n-m}}(A)}{n-m}$. The choice of optimal indices $i_1,...,i_{n-m}$ can now be posed as the type of problem considered in Lemma \ref{lemma:consecutive_numbers}. As the eigenvalues of $A$ are indexed in non-decreasing order, the optimal choice of $i_1,...,i_{n-m}$ must be $n-m$ consecutive integers $l,l+1,,...,l+n-m-1$ according to the lemma.
\end{proof}

\begin{theorem}\label{thm:lowest_dist_tr_tr}
    Let $A$ be an $n \times n$ real positive definite matrix with eigenvalues $\lambda_1(A),...,\lambda_n(A)$ and eigenvectors $v_1(A),...,v_n(A)$. The optimal solutions to
    \begin{equation*}
    \begin{aligned}
    & \underset{X}{\text{minimize}}
    & & \tr(XA^{-1}) + \tr(X^{-1}A) \\
    & \text{subject to}
    & & X \in {\mathcal L}_{m,n}, X \succ 0 \\
    \end{aligned}
    \end{equation*}
    for which $v_k(X) = v_k(A) \ \forall \ k \in \{1,...,n\}$ are precisely those such that \\ $\lambda_i(X) = \lambda = \left(\frac{\lambda_l(A)+\lambda_{l+1}(A)+...+\lambda_{l+n-m-1}(A)}{\lambda_l(A)^{-1}+\lambda_{l+1}(A)^{-1}+...+\lambda_{l+n-m-1}(A)^{-1}}\right)^{(1/2)}$ if $i \in \{l,...,l+n-m-1 \}$ and $\lambda_i(A)$ otherwise, and $l$ is an integer that minimizes $\sum_{i=l}^{l+n-m-1}\frac{\lambda_i(A)}{\lambda} + \frac{\lambda}{\lambda_i(A)}$.
\end{theorem}

\begin{proof}
When $X$ and $A$ have the same eigenvectors,
\begin{equation*}
    \tr(XA^{-1}) + \tr(X^{-1}A) = \sum_{i=1}^n \frac{\lambda_i(X)}{\lambda_i(A)} + \frac{\lambda_i(A)}{\lambda_i(X)}.
\end{equation*}

As in the proof of Theorem~\ref{thm:lowest_dist_frobenius}, $n-m$ eigenvalues of $X$ must be set to one and the same value, and it is optimal to set the other eigenvalues equal to the corresponding eigenvalues of $A$, as the function $x + \frac{1}{x}$ has a unique global minimum in its domain for $x = 1$. Let $i_1,...,i_{n-m}$ be the indices of those eigenvectors of $A$ and $X$ for which it is optimal to set the corresponding eigenvalues of $X$ equal to each other, and let all those eigenvalues be equal to $\lambda$. Then, the part of $\tr(XA^{-1}) + \tr(X^{-1}A)$ that we need to optimize with respect to $\lambda$ is $\sum_{k=1}^{n-m} \frac{\lambda}{\lambda_{i_k}(A)} + \frac{\lambda_{i_k}(A)}{\lambda} $. Since $\lambda$ is the value that minimizes this sum, the derivative of the sum with respect to $\lambda$ must be 0, so $\sum_{k=1}^{n-m} \frac{1}{\lambda_{i_k}(A)} - \frac{\lambda_{i_k}(A)}{\lambda^2} = 0 \Leftrightarrow \lambda = \left(\frac{\lambda_{i_1}(A)+\lambda_{i_2}(A)+...+\lambda_{i_{n-m}}(A)}{\lambda_{i_1}(A)^{-1}+\lambda_{i_2}(A)^{-1}+...+\lambda_{i_{n-m}}(A)^{-1}}\right)^{(1/2)}$. The choice of optimal indices $i_1,...,i_{n-m}$ can now be posed as the type of problem considered in Lemma \ref{lemma:consecutive_numbers}. As the eigenvalues of $A$ are indexed in non-decreasing order, the optimal choice of $i_1,...,i_{n-m}$ must be $n-m$ consecutive integers $l,l+1,,...,l+n-m-1$ according to the lemma.
\end{proof}

As we see, solving the matrix optimization problems in theorems~\ref{thm:lowest_dist_tr_logdet1}-~\ref{thm:lowest_dist_tr_tr} amounts to finding a consecutive sequence of the eigenvalues of the matrix $A$ to be approximated, computing some form of mean of these eigenvalues and setting the eigenvalues in the sequence to that mean, while leaving all other properties of $A$ unchanged. Both the means and the matrix dissimilarity measures expressed as a function of the eigenvalues can be computed recursively when one eigenvalue is added and one removed from a sequence. Thus, finding the optimal sequence of eigenvalues to modify can be done with computational complexity $O(n+m)$. 

Finding the eigenvalue decomposition of an $n \times n$ real symmetric has computational complexity $O(n^3)$. This dominates the other computational costs associated with solving the limited memory approximation problems that we have considered in this subsection. To reduce the computational cost further, the matrices to be approximated must have such a structure that their eigendecomposition is available at a reduced cost. This will be the case for the matrix approximation problems considered in the next subsection.

\subsection{Reducing limited memory matrices}\label{subsec:reduction}

Apart from taking up limited memory, the matrices in $\mathcal{L}_{m,n}$ are much less costly to perform computations with than general $n\times n$ matrices. In particular, computations involving their eigenvalues and eigenvectors can be done much more efficiently, something that was used to great advantage by e.g. Burdakov et. al. to derive trust-region methods \cite{Combining_TR_and_LQN}. When $A \in \mathcal{L}_{m,n}$, its eigenvalue decomposition can be obtained at a computational cost that scales as $O(nm^2 + m^3)$. The linear scaling in $n$ is important, as it means that very large matrix dimensions can be handled as long as $m$ is not too large. To see how the eigenvalue decomposition of $A$ can be obtained, note that a general matrix in $\mathcal{L}_{m,n}$ can be written as 

\begin{equation*}
    A = \alpha I + U C U^T,
\end{equation*}
where $U \in \mathbb{R}^{n \times m}$ and $C \in \mathbb{R}^{m \times m}$. By performing a QR-factorization of $U$, it can be written as $U = QR$, where $Q \in \mathbb{R}^{n,m}$ is unitary and $R\in \mathbb{R}^{m\times m}$ is upper triangular. The matrix $M:= RCR^T$, it is an $m \times m$ real symmetric matrix, hence it has an eigenvalue decomposition $M = V \Lambda V^T$ with the columns of $V$ being orthonormal eigenvectors, that can be obtained for a given precision with complexity $O(m^3)$. Putting this together, we have that
\begin{equation}\label{eqn:matrix_structure}
    A = \alpha I + U C U^T = \alpha I + QRCR^TQ^T = QV \Lambda V^TQ^T.
\end{equation}
From this, we see that every column of $QV$ is an eigenvector of $A$ corresponding to an eigenvalue equal to the sum of $\alpha$ and one of the diagonal elements of $\Lambda$. In addition, every vector that is orthogonal to all columns of $QV$ is an eigenvector of $A$ corresponding to the eigenvalue $\alpha$.
The latter kind of eigenvector is only implicitly defined. If needed for explicit calculations, one could randomly generate a vector and use e.g. Gram-Schmidt orthogonalization to make it orthogonal to the explicitly known eigenvectors. However, when $n$ is large, such computations would be costly, and desirable to avoid if possible. 

A specific application of the matrix optimization problems considered in the previous subsection is to derive a type of limited-memory quasi-Newton method, which alternates quasi-Newton updates with memory reductions. For this, it is relevant to note that performing a quasi-Newton update on a given matrix typically implies adding a low-rank matrix to it. Assuming that the low-rank matrix is of rank $l$, and the matrix to be updated is in $\mathcal{L}_{m,n}$, the updated matrix will be in $\mathcal{L}_{m+l,n}$. An example of this will be shown in Section~\ref{sec:NovelQN}. Then, to keep the memory requirements limited, one can replace the updated matrix with its best approximation in $\mathcal{L}_{m,n}$, based on minimizing some matrix dissimilarity measure. In the remainder of the article, we will call this "reducing" the matrix, and the resulting matrix the "reduced" matrix. 

When reducing a matrix, one would want to avoid the previously mentioned issue of having to explicitly generate any of the eigenvectors that were only implicitly defined before the reduction. Fortunately, for three of the matrix dissimilarity measures that we have considered, it is sufficient to work with only the explicitly stored eigenvectors, provided that $n$ is sufficiently large compared to $m$ and $l$. Another way to express this is that if $A \in \mathcal{L}_{m+l,n}$ is the matrix to be reduced, and it has the structure described in Equation~\eqref{eqn:matrix_structure},
all the eigenvalues corresponding to eigenvectors in the implicitly defined eigenspace will be part of the consecutive sequence of eigenvalues averaged to produce the eigenvalue of multiplicity at least $n-m$ of the reduced matrix. We demonstrate this property in the proofs of the three theorems below.

\begin{theorem}\label{thm:l2 n-m-l}
    Let $n > 2m + l$, $A \in {\mathcal L}_{m+l,n}$ and $\alpha$ be the eigenvalue of $A$ that has multiplicity at least $n-m-l$. The nearest limited memory matrix problem
    \begin{equation*}
    \begin{aligned}
    & \underset{X}{\text{minimize}}
    & & \| X - A \|_2 \\
    & \text{subject to}
    & & X \in {\mathcal L}_{m,n} \\
    \end{aligned}
    \end{equation*}
    has an optimal solution $X^{\star}$ such that its eigenvalue of multiplicity at least $n-m$ is the arithmetic mean of the maximum and minimum eigenvalue in a sequence of $n-m-l$ consecutive eigenvalues of $A$ equal to $\alpha$ and no more than $l$ eigenvalues of $A$ distinct from $\alpha$.  
\end{theorem}

\begin{proof}
    By Theorem \ref{thm:lowest_dist_l2}, the problem has an optimal solution where the eigenvalue of multiplicity $n-m$ is the mean of the maximum and minimum eigenvalue in a sequence of $n-m$ consecutive eigenvalues of $A$. An optimal solution can be found by setting $n-m$ such sequential eigenvalues of $A$ to that average while leaving the other eigenvalues and the eigenvectors unchanged. Assume that modifying a particular sequence of eigenvalues in this way results in an optimal solution. When $n - m > m + l$, that sequence must contain at least one eigenvalue equal to $\alpha$. If $\alpha$ is not the highest or lowest eigenvalue in the sequence, all eigenvalues equal to $\alpha$ must be contained in the sequence and the theorem holds. Otherwise, either $n-m-l$ eigenvalues equal to $\alpha$ are contained in the sequence, or it is possible to get a sequence with an absolute difference between the highest and lowest eigenvalue that is equal to or lower than that of the current sequence by replacing either the highest or the lowest eigenvalue in the sequence with one of the eigenvalues equal to $\alpha$ that are not in the sequence, contradicting the optimality of the sequence.
\end{proof}

\begin{theorem}\label{thm:tr2 n-m-l}
    Let $n \geq 3m + 2l$, $A \in {\mathcal L}_{m+l,n}$ and $\alpha$ be the eigenvalue of $A$ that has multiplicity at least $n-m-l$. The nearest limited memory matrix problem
    \begin{equation*}
    \begin{aligned}
    & \underset{X}{\text{minimize}}
    & & \tr(XA^{-1}) + \tr(X^{-1}A) \\
    & \text{subject to}
    & & X \in {\mathcal L}_{m,n} \\
    \end{aligned}
    \end{equation*}
     has an optimal solution $X^{\star}$ such that its eigenvalue of multiplicity at least $n-m$ is the geometric mean of the arithmetic and harmonic means of all eigenvalues in a sequence of $n-m-l$ consecutive eigenvalues of $A$ equal to $\alpha$ and no more than $l$ eigenvalues of $A$ distinct from $\alpha$.
\end{theorem}
\begin{proof}
    By Theorem \ref{thm:lowest_dist_tr_tr}, the problem has an optimal solution where the eigenvalue of multiplicity $n-m$ is the geometric mean of the arithmetic and harmonic means of all eigenvalues in a sequence of $n-m$ consecutive eigenvalues of $A$. An optimal solution can be found by setting $n-m$ such sequential eigenvalues of $A$ to that value while leaving the other eigenvalues and the eigenvectors unchanged. Assume that modifying the sequence of eigenvalues starting with index $k$ in this way results in an optimal solution. The value of the objective function is then 
    \begin{equation*}
        2m + 2 \sqrt{\sum_{i=k}^{k+n-m-1} \lambda_i \sum_{i=k}^{k+n-m-1} \frac{1}{\lambda_j}}.
    \end{equation*}
     If $\alpha$ is not the highest or lowest eigenvalue in the sequence, all eigenvalues equal to $\alpha$ must be contained in the sequence and the theorem holds. It also holds when both the highest and lowest eigenvalues are equal to $\alpha$. We thus only need to consider the case where all eigenvalues in the sequence not equal to $\alpha$ are either all higher or all lower than $\alpha$. The argument for both cases is analogous, differing only by the values of certain indices and the direction of certain inequalities, so without loss of generality, we can assume that they are all higher. Let the highest eigenvalue contained in the sequence be equal to $\lambda > \alpha$. If $k=1$, all eigenvalues equal to $\alpha$ are contained in the sequence and the theorem holds, so to derive a contradiction, we assume that $k > 1$ and that $\lambda_{k-1}(A)=\alpha$. Since the sequence we average over is optimal, the difference

    \begin{equation}\label{eqn:objdiff}
    \begin{aligned}
        &\sum_{i=k-1}^{k+n-m-2} \lambda_i \sum_{j=k-1}^{k+n-m-2} \frac{1}{\lambda_j} - \sum_{i=k}^{k+n-m-1} \lambda_i \sum_{j=k}^{k+n-m-1} \frac{1}{\lambda_j}\\
        &= \sum_{i=k-1}^{k+n-m-2} \frac{\lambda_i}{\alpha} + \frac{\alpha}{\lambda_i}  - \sum_{i=k}^{k+n-m-1} \frac{\lambda}{\lambda_j}+\frac{\lambda_j}{\lambda}
    \end{aligned}
    \end{equation}

    should be positive. But what we will show is that even in a "worst-case scenario", it is not. To maximize the right-hand side of Equation~\eqref{eqn:objdiff}, the eigenvalues not necessarily equal to $\alpha$, which by our assumption lie in the interval $[\alpha,\lambda]$, should be equal to $\lambda$. In that case, all terms in both sums will be equal to either 2 or $\lambda/\alpha + \alpha/\lambda > 2$. However, when $n \geq 3m + 2l$, at least half of the eigenvalues in the sequence $\lambda_k,...,\lambda_{k+n-m-1}$ must be equal to $\alpha$. Then, the left sum will have more terms equal to 2 than the right sum, and hence be lower, rendering the difference between the sums negative. This leads to a contradiction.
\end{proof}

\begin{theorem}\label{thm:Frobenius n-m-l}
    Let $n \geq 2m + 2l$, $A \in {\mathcal L}_{m+l,n}$ and $\alpha$ be the eigenvalue of $A$ that has multiplicity at least $n-m-2$. The nearest limited memory matrix problem
    \begin{equation*}
    \begin{aligned}
    & \underset{X}{\text{minimize}}
    & & \| X - A \|_F \\
    & \text{subject to}
    & & X \in {\mathcal L}_{m,n} \\
    \end{aligned}
    \end{equation*}
    has an optimal solution $X^{\star}$ such that its eigenvalue of multiplicity at least $n-m$ is the arithmetic mean of $n-m-l$ eigenvalues of $A$ equal to $\alpha$ and no more than $l$ eigenvalues of $A$ distinct from $\alpha$.  
\end{theorem}
\begin{proof}
By Theorem \ref{thm:lowest_dist_frobenius}, the problem has an optimal solution where the eigenvalue of multiplicity $n-m$ is the mean of $n-m$ consecutive eigenvalues of $A$. An optimal solution can be found by setting $n-m$ of the eigenvalues of $A$ to their average while leaving the other eigenvalues and the eigenvectors unchanged. We shall prove that there will always exist an optimal sequence containing $n-m-2$ eigenvalues equal to $\alpha$ for the considered case. The proof has two parts: First, we determine the worst case locations of the eigenvalues of $A$, under some constraints that can be imposed without loss of generality. We then show that using the sequence with all eigenvalues equal to $\alpha$ is optimal even in this case. Since we am only considering consecutive sequences, We can rule out any sequence that does not contain $n-m-l$ eigenvalues equal to $\alpha$, but does contain eigenvalues both lesser and greater than $\alpha$. Without loss of generality, the considered sequences can be restricted to ones where $\alpha$ is the minimum eigenvalue in the sequence. The eigenvalues that are not going to be included in a sequence can without loss of generality be assumed to be greater than $\alpha$ too. The case where there are $n-m$ eigenvalues equal to $\alpha$ is trivial and does not need to be considered. Denote the eigenvalues of $A$ by $\lambda_1,...,\lambda_n$, where $\lambda_1=\lambda_2=...=\lambda_{n-m-l} = \alpha$,
$\lambda_{j} \in [\alpha,\lambda_{n-m}] \ \forall j \in \{n-m-l+1,...,n-m-1\}$, $\lambda_{n-m} = \lambda > \alpha$ and $\lambda_{j} \geq \lambda \ \forall j \in \{n-m+1,...,n\}$. Let $F(k,\lambda_1,...,\lambda_n)$ be the Frobenius norm of the difference between $X$ and $A$ when $X$ is obtained from $A$ by setting all eigenvalues in the sequence $\lambda_{k},...,\lambda_{k+n-m-1}, \ k \in \{1,...,m+1\}$ to their mean but leaving all other properties of $A$ unchanged. Then, 
\begin{align*}
    F(k,\lambda_1,...,\lambda_n) = \sum_{i=k}^{k+n-m-1} (\lambda_i - \bar \lambda_k)^2 &,\quad \bar \lambda_k = \frac{1}{n-m} \sum_{i=k}^{k+n-m-1} \lambda_i.
\end{align*}
Furthermore, let $D(k_1,k_2,\lambda_1,...,\lambda_n) = F(k_1,\lambda_1,...,\lambda_n) - F(k_2,\lambda_1,...,\lambda_n)$. Letting $\alpha$ and $\lambda$ and $k > 1$ be fixed, the worst-case locations of the eigenvalues are given by the solution of the following optimization problem:
\begin{equation*}
\begin{aligned}
& \underset{\lambda_{1},...,\lambda_n}{\text{minimize}}
& & D(k,1,\lambda_1,...,\lambda_{n}) \\
& \text{subject to}
& & \lambda_j = \alpha, &\forall j \in \{1,...,n-m-l\}, \\
& & & \lambda_{j} \in [\alpha,\lambda], &\forall j \in \{n-m-l+1,...,n-m-1\}, \\
& & & \lambda_{n-m} = \lambda, \\
& & & \lambda_j \geq \lambda, &\forall j \in \{n-m+1,...,k+n-m-1\},\\
& & & \lambda_j \leq \lambda_{j+1}, &\forall j \in \{1,...,n-1\}.
\end{aligned}
\end{equation*}
Consider the derivative of $F(k,\lambda_1,...,\lambda_n)$ with respect to $\lambda_j$ for $j \in \{k,...,
k+n-m-1\}$. 
\begin{equation*}
\begin{aligned}
\frac{\partial F(k,\lambda_1,...,\lambda_n)}{\partial \lambda_j} & = 2 \frac{n-m-1}{n-m}(\lambda_j \\
& - \bar \lambda_k) - 2\frac{1}{n-m}\underset{i \neq j}{\sum_{i \in \{k,...,k+n-m-1\}}} (\lambda_i - \bar \lambda_k) = 2 (\lambda_j - \bar \lambda_k).
\end{aligned}
\end{equation*}
From this, it follows that 
\begin{equation*}
\frac{\partial D(k,1,\lambda_1,...,\lambda_n)}{\partial \lambda_j} = \begin{cases}
2(\bar \lambda_1 - \lambda_j), \quad j \in \{1,...,k-1\}, \\
2(\bar \lambda_1 - \bar \lambda_k), \quad j \in \{k,n-m\},\\
2(\lambda_j - \bar \lambda_k), \quad j \in \{n-m+1,k+n-m-1\},\\
0, \quad j \in \{ k+n-m,n\}.
\end{cases}
\end{equation*}

The derivative of $D(k,1,\lambda_1,...,\lambda_n)$ with respect to any of the eigenvalues constrained to lie between $\alpha$ and $\lambda$ is $2 (\bar \lambda_1 - \bar \lambda_k)$, and since $\bar \lambda_1 - \bar \lambda_k <0$ for all combinations of eigenvalues satisfying the constraints, $D(k,1,\lambda_1,...,\lambda_n)$ is minimized with respect to these variables by setting them to their maximum value according to the constraints, $\lambda$. The derivative with respect to $\lambda_j$ for $j \in \{n-m-1,...,n\}$ is either $0$ or $2(\lambda_j - \bar \lambda_k)$, implying that $D(k,1,\lambda_1,...,\lambda_n)$ is minimized with respect to these variables by choosing them to be as close to $\bar \lambda_k$ as possible when satisfying the constraints. This is achieved by setting them equal to $\lambda$ too. Thus, for a fixed $\alpha$, $k$ and $\lambda$, one worst case scenario is a case where $n-m-l$ eigenvalues are equal to $\alpha$ and $m + l$ eigenvalues are equal to $\lambda$. In the sequence $\lambda_k,...,\lambda_{k+n-m-1}$, there will then be $n-m-l-k+1$ eigenvalues equal to $\alpha$ and $k-1+l$ eigenvalues equal to $\lambda$. Then,
\begin{equation}\label{eqn:Fnorm_worst_case}
    F(k,\alpha,...,\alpha,\lambda,...,\lambda) = \frac{(n-m-l-k+1)(l+k-1)}{n-m}(\lambda-\alpha)^2.
\end{equation}
The only factor in Equation~\eqref{eqn:Fnorm_worst_case} that varies with $k$ is the numerator of the fraction, so it suffices to check that for $k \in \{1,...,m+1\}$, it is minimized for $k=1$. Since it is a quadratic function of $k$ with a negative second-order coefficient, its minimum on the set must be at either $1$ or $m+1$. For $k=1$, the numerator is $(n-m-l)l$, and for $k=m+1$, it is $(n-2m-l)(m+l)$. The difference is $(n-2m-l)(m+l) - (n-m-l)l = (n-2m-2l)m$, and by the assumption that $n-m-l\geq m + l, (n-2m-2l)m \geq 0$. 
\end{proof}

While the theorems above show the convenience of reducing matrices by using the matrix dissimilarity measure each of them applies to, no corresponding theorems exist for the Stein loss and inverse Stein loss. A particular form of counterexample to any such theorem is when $A$ has $n-m-l$ eigenvalues equal to $\alpha$ and $m+l$ eigenvalues equal to $\lambda$. By Theorem~\ref{thm:lowest_dist_tr_logdet1}, when minimizing $\tr(X A^{-1}) - \log \det (X A^{-1})$ it is optimal to form $X \in \mathcal{L}_{m,n}$ by setting all eigenvalues in a consecutive sequence to their harmonic mean. If $\lambda/\alpha$ is small enough, choosing the sequence that includes all eigenvalues equal to $\alpha$ will always lead to higher matrix dissimilarity than choosing the sequence that contains one more eigenvalue equal to $\lambda$. Similarly, Theorem~\ref{thm:lowest_dist_tr_logdet2}, says that when minimizing $\tr(X^{-1} A) - \log \det (X^{-1} A)$, it is optimal to set a sequence of eigenvalues of $A$ to their arithmetic mean. If $\lambda/\alpha$ is large enough, choosing the sequence containing all eigenvalues equal to $\alpha$ will lead to greater matrix dissimilarity than choosing the one with one more eigenvalue equal to $\lambda$.

\section{Novel limited memory quasi Newton algorithms}\label{sec:NovelQN}

Having considered general nearest matrix optimization problems, we will now focus on how they can be applied to implement limited memory adaptations of quasi-Newton methods. The basic idea of our methods will be to start with a Hessian approximation $B_k\in {\mathcal L}_{m,n}$, use the most recent iterate and gradient in one of many Broyden-class updates to form a matrix  $\hat{B}_{k+1}\in {\mathcal L}_{m+2,n}$, use this updated Hessian approximation to generate the next iterate, and then reduce the Hessian approximation back to $B_{k+1}\in {\mathcal L}_{m,n}$ before the next iteration. If the Hessian approximation in iteration $k$ of a quasi-Newton method belongs to $\mathcal{L}_{m,n}$, it can be written as
\begin{equation*}
    B_k = \alpha_k I + U_k C_k U_k^T.
\end{equation*}
The update formula for a general Broyden class method, with the commonly utilized notation of $s_k$ for the difference between two subsequent iterates and $y_k$ for the corresponding gradients, can when the Broyden class interpolation parameter $\phi_k \neq 1$, be expressed as
\begin{equation}\label{eqn:Broyden_short}
   B_{k+1} = B_k +  \frac{\phi_k - 1}{s_k^TB_ks_k} c_{k} c_{k}^T + \left(1 - \frac{\phi_k s_k^TB_ks_k}{s_k^Ty_k(\phi_k - 1)} \right) \frac{y_k y_k^T}{s_k^Ty_k},
\end{equation}
where
\begin{equation*}
    c_{k} =  B_ks_k - \frac{s_k^TB_ks_k \phi_k}{(\phi_k-1)s_k^Ty_k}y_k.
\end{equation*}
Note that the low-rank plus shift structure is preserved and that the rank of the low-rank term increases by at most two. The update can be written as
\begin{equation}\label{eqn:Broyden_compact}
\begin{aligned}
    B_{k+1} & = \alpha_k I + \begin{bmatrix} U_k & c_k & y_k \end{bmatrix} \begin{bmatrix} C_k & 0 & 0 \\ 0 & \frac{\phi_k - 1}{s_k^TB_ks_k} & 0 \\ 0 & 0 &  1 - \frac{\phi_k s_k^TB_ks_k}{s_k^Ty_k(\phi_k - 1)} \end{bmatrix} \begin{bmatrix} U_k^T \\ c_k^T \\ y_k^T \end{bmatrix} \\
    & := \alpha_{k} I + (U_{k+1})^T C_{k+1} U_{k+1}.
\end{aligned}
\end{equation}
The DFP update, for which $\phi_k=1$, must be considered separately since the Broyden class update formula above is not defined in this case. Hence, we can not identify $C_{k+1}$ and $U_{k+1}$ as above, but can instead choose them as
\begin{equation}\label{eqn:U+_alt}
    U_{k+1} = \begin{bmatrix} U_k & B_{k} s_k & y_k
    \end{bmatrix}
\end{equation}
and
\begin{equation}\label{eqn:A+_alt}
    C_{k+1} = \begin{bmatrix} C_k & 0 & 0 \\ 0 & 0 & -\frac{1}{s_k^Ty_k} \\ 0 & -\frac{1}{s_k^Ty_k} &  \frac{1}{s_k^Ty_k}\left(\frac{s_k^T B_{k} s_k}{s_k^Ty_k} +1 \right) \end{bmatrix}.
\end{equation}

We see that $ \widehat B_{k+1}$, when computed from $B_{k}\in \mathcal{L}_{m,n}$ using a Broyden class formula, still ends up in $\mathcal{L}_{m+2,n}$. As we outlined in Section~\ref{subsec:reduction}, a limited memory quasi-Newton method can be derived by reducing the matrix back to ${\mathcal L}_{m,n}$ by solving one of the matrix optimization problems defined in Section~\ref{sec:NLMMP}. This requires us to perform an eigenvalue decomposition of the updated matrix, but this can be done efficiently using the techniques discussed in Section~\ref{subsec:reduction}, with complexity $O(m^2n + m^3)$. A more detailed discussion of this is given in the appendix, Section~\ref{sec:Eigendecomposition}. We can also note that some variations are possible: Instead of reducing the Hessian approximation every iteration, we could use greater memory reductions with a lower frequency. In any case, the basic principles are the same. 

Although algorithms in the proposed framework require more computations per iteration than the usual limited-memory Broyden-class methods, we want to emphasize that in the intended setting, it is the evaluations of the objective function and its gradient that take up most of the time and that an increase in computations in exchange for a lower number of function and gradient evaluations is a desirable trade-off. The computational complexity is similar to the worst-case complexity of the AGG-BFGS algorithm of Berahas et al. \cite{AGG_BFGS}, which with limited memory can reproduce the full memory BFGS iterations when the differences in iterates lie in a low-dimensional subspace. This will also be the case for a limited memory version of BFGS that applies matrix reduction to limit its memory requirements if it has the same amount of memory available as AGG-BFGS that reproduces full memory BFGS exactly (i.e. can store at least twice as many vectors as the dimension of the low-dimensional subspace). However, an important difference is that AGG-BFGS gets rid of old curvature information when an iterate displacement $s_k$ does not seem to lie in the span of previous ones, while the algorithms in our framework apply an optimality principle to minimize the necessary change in the matrix to be able to store it with limited memory. 

\subsection{Limitations of scope}\label{sec:limitations}

To evaluate every possible quasi-Newton update with every possible nearest-matrix formulation would be a gargantuan task. We therefore limit the scope of our numerical experiments with quasi-Newton methods to one based on the BFGS update ($\phi_k = 0$ in the Broyden class update). As BFGS is one of the most widely used Broyden-class methods and is considered to be among the best, we believe that it makes sense to focus on it. We further limit the type of matrix reduction to the minimization of $l^2$ and Frobenius norms of matrix differences. These matrix dissimilarity measures can in most applications, where $n \gg m$, be minimized without the need to explicitly generate any eigenvectors corresponding to $\alpha_k$. This follows from theorems~\ref{thm:l2 n-m-l} and~\ref{thm:Frobenius n-m-l} respectively. We take particular interest in comparing their difference in performance based on a qualitative difference between them. Minimizing the $l^2$-norm makes the algorithm completely ``forget´´ some eigenvalue information in each iteration, in the sense that while only two eigenvalues determine the value of $\alpha_{k+1}$, the values of the eigenvalues between them are discarded without influencing any property of the updated matrix. In contrast, all eigenvalues changed when minimizing the Frobenius norm influence the value of $\alpha_{k+1}$, meaning that they continue to influence future iterations. A hypothesis based on this is that using the Frobenius norm is more suitable when minimizing functions with a constant, or near-constant Hessian, while the $l^2$-norm will be better suited to functions for which the Hessian changes more between iterations.  

One more choice that we make in our implementations is to combine the limited memory quasi-Newton method with a trust-region method. This is due to the synergy that exists between these types of methods. The memory-reduction part of the quasi-Newton method relies on having an eigendecomposition of the matrix to be reduced available, and this factorization of the matrix is also very advantageous to have when solving trust-region problems. We elaborate on the specifics of this below and refer to the appendix, Section~\ref{sec:TR}, for an overview of trust-region methods.

\subsection{Efficient implementation of a trust-region algorithm}

To summarize, the algorithms that we implement carry out the following three steps throughout an iteration:
\begin{enumerate}
    \item Update the Hessian approximation $B_k$ according to the BFGS formula to obtain $\widehat{B}_{k+1}.$
    \item Solve a trust-region problem with the updated Hessian approximation $\widehat{B}_{k+1}$ using Algorithm~\ref{alg:tr_sol} (shown in the appendix).
    \item Reduce the memory requirements of the Hessian approximation $\widehat{B}_{k+1}$ by replacing it with the closest matrix, either in $l^2$-norm or Frobenius norm, with $B_{k+1}$ in $\mathcal{L}_{m,n}$. 
\end{enumerate}
The trust region problem is solved using $\widehat{B}_{k+1}$ since this matrix is guaranteed to satisfy the secant equation. Because the eigendecomposition is useful in both steps 2 and 3, it is partially carried out before the trust region problem is solved. As we discuss in Section~\ref{sec:Eigendecomposition} of the appendix, we can write matrices in $\widehat B_{k+1}$ on the form
\begin{equation}
    \widehat B_{k+1} = \alpha_k I +  Q_{k+1} V_{k+1} \Lambda_{k+1} V_{k+1}^T Q_{k+1}^T, 
\end{equation}
where $\Lambda_{k+1} \in \mathbb{R}^{m \times m}$ is diagonal and $Q_{k+1} \in \mathbb{R}^{n\times m}$ and $V_{k+1} \in \mathbb{R}^{m\times m}$ are orthonormal matrices. It is evident from this factorization that the eigenvectors of $\widehat B_{k+1}$ are the columns of $E_{k+1} := Q_{k+1} V_{k+1}$, as well as any vectors orthogonal to them. However, it is not necessary to compute them explicitly to solve the trust-region problem. One way to take advantage of parallel computation capabilities is to delay the computation of $E_{k+1}$ until the trust-region subproblem is solved. By doing so, a new gradient can be computed at the same time as $E_{k+1}$, saving total computation time. Since forming $E_{k+1}$ has computational complexity $O(nm^2)$, the per-iteration time can be reduced considerably by performing this operation in parallel with other time-consuming operations. 

The main synergy between our limited memory methods and the trust-region framework (described in Section~\ref{sec:TR} of the appendix) is that the iterations of the trust-region subroutine can be carried out without explicitly forming any $n$-dimensional vectors. A trust-region method works by finding an optimal increment $p$ to add to the current iterate. For the trust-region problem that we consider, the conditions for $p$ to be an optimal solution is that $(\widehat B_{k+1} + \sigma I)p = -g_k$, where $g_k$ is the current gradient, and $(\|p\|-d_k) \sigma = 0$, for some real number $\sigma \geq 0$. The trust region algorithm works by finding tentative values of $\sigma$, and defines a $p$ for every such value in its iterations. It is however not necessary to explicitly form $p$ until it has been found to satisfy the stopping criterion for the trust-region subroutine. Only $\|p\|_2$ and $p^T(\widehat B_{k+1}+\sigma I)^{-1}p$ are needed to obtain the next iterate, and the structure of $\widehat B_{k+1}$ can be leveraged to compute them efficiently. A description of how to do this now follows.

Let $h_k = V_{k+1}^TQ^T_{k+1}g_k$, and let $g_k^{\bot}$ be the orthogonal component of $g_k$ relative to the range of $Q_{k+1}$, and thereby also the orthogonal component relative to the range of $Q_{k+1} V_{k+1}$. These quantities can be computed by first introducing $\widehat h_k = Q^T_{k+1}g_k$ and then evaluating $h_k = V_{k+1}^T \widehat h_k$, and $g_k^{\bot} = g_{k} - \sum_{i=1}^{m+2} \widehat h_{k}^{(i)} q_{i,k+1}$, where $ \widehat h_{k}^{(i)}$ is element $i$ of the vector $ \widehat h_{k}$ and $q_{i,k+1}$ is column $i$ of $Q_{k+1}$. In the parts of the algorithm where iterative computations are used to find $\sigma$ that fulfill the optimality conditions, $p$ is uniquely defined by $\sigma$ and can be written as a function $p(\sigma)$. We have that
\begin{equation*}
\begin{aligned}
    p(\sigma) & = -(\widehat B_{k+1} + \sigma I )^{-1} g_{k} \\
    & = - ((\alpha_{k} + \sigma)I + U_{k+1}V_{k+1}\Lambda_{k+1}V_{k+1}^TQ^T_{k+1})^{-1} (g_{k}^{\bot} + Q_{k+1}V_{k+1} h_{k}) \\
    & = -\frac{1}{\alpha_{k} + \sigma}g_{k}^{\bot} - Q_{k+1}V_{k+1} (I (\sigma + \alpha_{k}) + \Lambda_{k+1})^{-1} h_k.
\end{aligned}\label{eqn:psigma}
\end{equation*}
Hence,
\begin{equation*}\label{eqn:pnorm}
\begin{aligned}
    & \|p(\sigma) \|_2 = \sqrt{\frac{\|g_{k}^{\bot}\|_2^2}{(\alpha_{k} + \sigma)^2} + h_k^T ((\sigma + \alpha_{k})I + \Lambda_{k+1})^{-2} h_k}, \\
    & p(\sigma)^T (\widehat B_{k+1}+\sigma I)^{-1}p(\sigma) = \frac{\|g_{k}^{\bot}\|_2^2}{(\alpha_{k} + \sigma)^3} + h_k^T ((\sigma + \alpha_{k})I + \Lambda_{k+1})^{-3} h_k.
\end{aligned}
\end{equation*}
We can note that once $h_k$ and $\|g_k^{\bot}\|$ have been computed (which can be done with complexity $O(nm)$), only $(m+2)$-dimensional vectors and $(m+2)\times(m+2)$ matrices are required to compute the above quantities. This makes the trust-region subroutine very efficient when $m \ll n$.

\subsection{A note about limits on eigenvalues to ensure convergence}

While limited memory quasi-Newton methods have had success in practice, to the best of our knowledge they have no theoretical advantages over accelerated variants of gradient descent when minimizing general twice-differentiable functions, apart from special cases such as when the objective function is quadratic. In general, their convergence is linear, and convergence proofs typically hinge on limits on the eigenvalues of the Hessian approximations. Given that the eigenvalues of $\widehat B_{k+1}$ are upper bounded throughout the iterations, we can make use of the trust-region framework described in Section~\ref{sec:TR} to demonstrate convergence. In particular, we direct readers interested in a convergence result to Theorem~\ref{thm:trust-convergence} and the discussion after it.

For the methods that we propose in this paper, one advantage is that the eigenvalues of the Hessian approximations are readily available, and their boundedness can be monitored. It is also possible to force boundedness of the eigenvalues by including it as a constraint when reducing the matrix, and Theorems~\ref{thm:eigenv} and~\ref{thm:eigenv_tr_logdet} ensure that it is sufficient to change only the eigenvalues to solve this optimization problem. When reducing matrices with respect to the $l^2$ norm, it is particularly easy to incorporate a constraint that all eigenvalues lie in an interval $[-\lambda_{\max},\lambda_{\max}]$, as it can be done by projecting the solution computed without this constraint onto the interval. For other matrix dissimilarity measures, first solving the nearest limited memory matrix problem without interval constraints on the eigenvalues and then projecting, is a heuristic method that does not always yield an optimal solution. We will not delve further into the problem of finding the nearest limited memory matrix with eigenvalue constraints here, as we do not make use of it in our implementation, but merely note that it is a possible way of ensuring theoretical convergence guarantees. 

\section{Numerical experiments}

In our numerical experiments, we test the aforementioned implementations of limited memory BFGS, where we keep the memory requirements limited by reducing the Hessian approximation with respect to the $l^2$-norm and the Frobenius norm. We call the former algorithm L2-BFGS and the latter LF-BFGS. While we have shown how to solve the associated matrix reduction problems analytically in exact arithmetic, it is also interesting to investigate how large the errors become when they are solved numerically, to see if a solver can be implemented in a numerically stable way. To that end, we replicate a test from \cite{AGG_BFGS} which studies an algorithm's numerical stability in reproducing the full memory BFGS Hessian approximation when that is possible. We then consider some data-driven test problems with relatively slowly varying Hessians, a class of optimization problems that motivated our developments. Firstly, we perform tests on a logistic regression problem with the publicly available data-set gisette. Thereafter, we run Monte Carlo simulations with randomly generated least-squares problems. Finally, to investigate how the methods perform on more general problems, we evaluate the proposed trust region methods on the CUTEst test suite. Throughout, we compare the convergence speed of L2-BFGS and LF-BFGS to MSS \cite{MSS}, a trust-region implementation of the conventional L-BFGS method. For the CUTEst problems, we also include a comparison with a line search implementation of L-BFGS. To avoid confusion of notation in these comparisons, We redefine the memory parameter $m$ in this section so that it matches standard L-BFGS terminology: $2m$ is here the maximum number of vectors stored to define the Hessian approximation at the beginning of an iteration, which corresponds to the number of stored curvature pairs $(s_k,y_k)$ for conventional L-BFGS.

\subsection{Curvature aggregation test}

The curvature information aggregation test is a Monte Carlo simulation performed for different memory parameters $m$ and numbers of variables $n$. In each realization, a sequence of curvature pairs $(s_1,y_1),...,(s_m,y_m)$ is generated by performing gradient descent with noisy gradients on a randomly generated quadratic optimization problem. Then, another curvature pair $(s_0,y_0)$ is generated in such a way that $s_0$ is a sum of $s_1,...,s_m$ with weights sampled from the standard normal distribution. The sequence of curvature pairs $(s_0,y_0),...,(s_m,y_m)$ is used to generate both the full memory and limited memory inverse Hessian approximations, using the inverse BFGS update. After this, a relative error, defined as the maximum difference between two corresponding elements in the inverse Hessian approximations, divided by the largest absolute value of an element of the full memory BFGS inverse Hessian approximation, is computed and recorded. The details of this procedure can be found in \cite{AGG_BFGS}. Essentially, the test scenario creates a situation in which a limited memory method could, in exact arithmetic, recreate the full-memory BFGS approximation. If this can be done numerically without large errors it means that the limited memory algorithm can achieve superlinear convergence in a situation where this is possible. 

Figure~\ref{fig:Agg_2_0} shows results for a test series with L2-BFGS, with boxplots of relative errors for different values of $n$ and $m$. Figure~\ref{fig:Agg_F_0} shows the corresponding plot for LF-BFGS, and we can see that the tests with both algorithms yielded similar results. Different values of the tolerance parameter $\nu$, as defined in Section~\ref{sec:Eigendecomposition} of the appendix, were tried. In the end, $\nu=0$ was found to yield the lowest relative errors. When comparing the results for L2-BFGS and LF-BFGS to the results for AGG-BFGS in the figure on page 21 in \cite{AGG_BFGS}, one can see that our algorithms have relative errors of a similar order of magnitude in most tests. In the worst-case scenario, where $m=n$, both LF-BFGS and L2-BFGS outperform AGG-BFGS. When $m<n$, the errors are slightly higher for LF-BFGS and L2-BFGS than AGG-BFGS. One difference between the AGG-BFGS, LF-BFGS and L2-BFGS methods is that AGG-BFGS stores unmodified curvature pairs up until the last iteration of the test, in which it is necessary to aggregate the stored pairs, while L2-BFGS and LF-BFGS modify the stored vectors in each iteration. This introduces rounding errors in every step and not just the last, explaining why L2-BFGS and LF-BFGS had higher relative errors than AGG-BFGS in most of the test runs, and that the difference is greatest when the test runs consist of many iterations. With this in mind, L2-BFGS and LF-BFGS outperforming AGG-BFGS when $m=n$ is particularly significant and speaks in favor of the numerical stability of their implementations. When $m=n$, $s_k$ and $y_k$ are guaranteed to eventually lie in the span of the stored eigenvectors even when rounding errors are present. As detailed in Section~\ref{sec:Eigendecomposition} of the appendix, our limited memory algorithms have a step where they explicitly detect and remove linearly dependent vectors after the inverse BFGS update, and that could explain why they perform especially well when $m = n$. 

\begin{figure}
    \centering
    \makebox[\textwidth][c]{\includegraphics[width=1.3\textwidth]{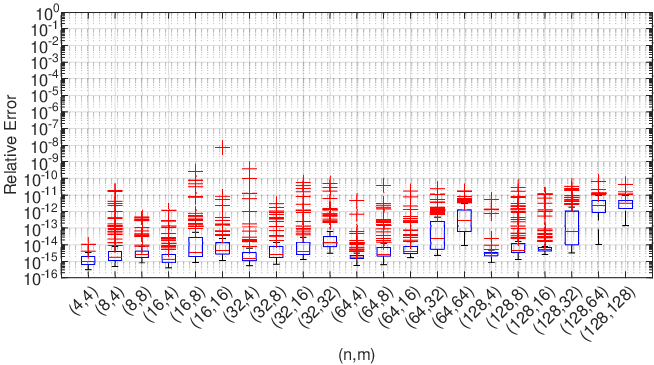}}
    \caption{Curvature aggregation test with the $l^2$ norm.}
    \label{fig:Agg_2_0}
\end{figure} 

\begin{figure}
    \centering
    \makebox[\textwidth][c]{\includegraphics[width=1.3\textwidth]{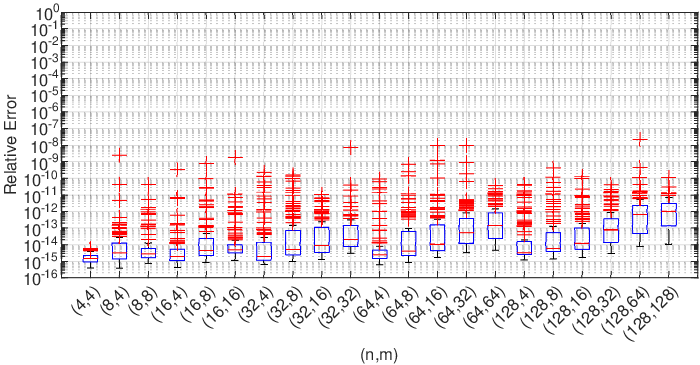}}
    \caption{Curvature aggregation test with the Frobenius norm.}
    \label{fig:Agg_F_0}
\end{figure} 

\subsection{Logistic regression}

In a supervised learning setting, data is available as pairs of features and labels, and the task is to find a function that can predict the label of a new sample, given its features. One statistical model for doing so is logistic regression, and in this test, we see how the algorithms perform when minimizing the objective function in a logistic regression problem. The purpose of this test, which was our first test chronologically, was to get an idea of whether our algorithms would perform well on a problem resembling those that could arise from real-world applications. To this end, we used the LIBSVM \cite{LIBSVM} dataset gisette, with 5000 variables and 6000 features. This dataset was first used in a challenge for support-vector machines, where the task was to distinguish handwritten digits.

To study how the performance of the algorithms depended on the parameter $m$, We ran them with $m=2,4,8$ and $16$. For each algorithm and each value of $m$, We checked the number of iterations necessary until $\|f(x_k)\|_2\leq 10^{-6}$. For each algorithm, we selected the value of $m$ with which it would fulfill this convergence criterion in the lowest number of iterations, and let the algorithms run for the maximum number of iterations necessary for them to satisfy this criterion. The objective function values for this test run are plotted in Figure \ref{fig:logistic_regression_gisette_comp}. We can see that L2-BFGS has the best performance in this test. Notably, for $m \in \{2,4,8,16\}$, it had the fastest fulfillment of the convergence criterion for $m=2$, satisfying it in 30 iterations. As a comparison, MSS had the best performance for $m=16$, satisfying the convergence criterion in 50 iterations, and LF-BFGS had it for $m=8$ and $m=16$, satisfying the criterion after 66 iterations. This result indicates that L2-BFGS with a low value of $m$ can be memory-efficient compared to trust-region implementations of traditional L-BFGS. It also points to a trade-off between retaining relevant and removing outdated curvature information, something that for some problems can be achieved better with a lower value of $m$. In addition, it aligns with the hypothesis stated in Section~\ref{sec:limitations}, that L2-BFGS performs better than LF-BFGS for problems with a Hessian that varies between iterations, due to how its matrix reduction procedure "forgets", rather than "aggregates" eigenvalues.

\begin{figure}
    \centering
    \includegraphics[width = 0.65 \textwidth]{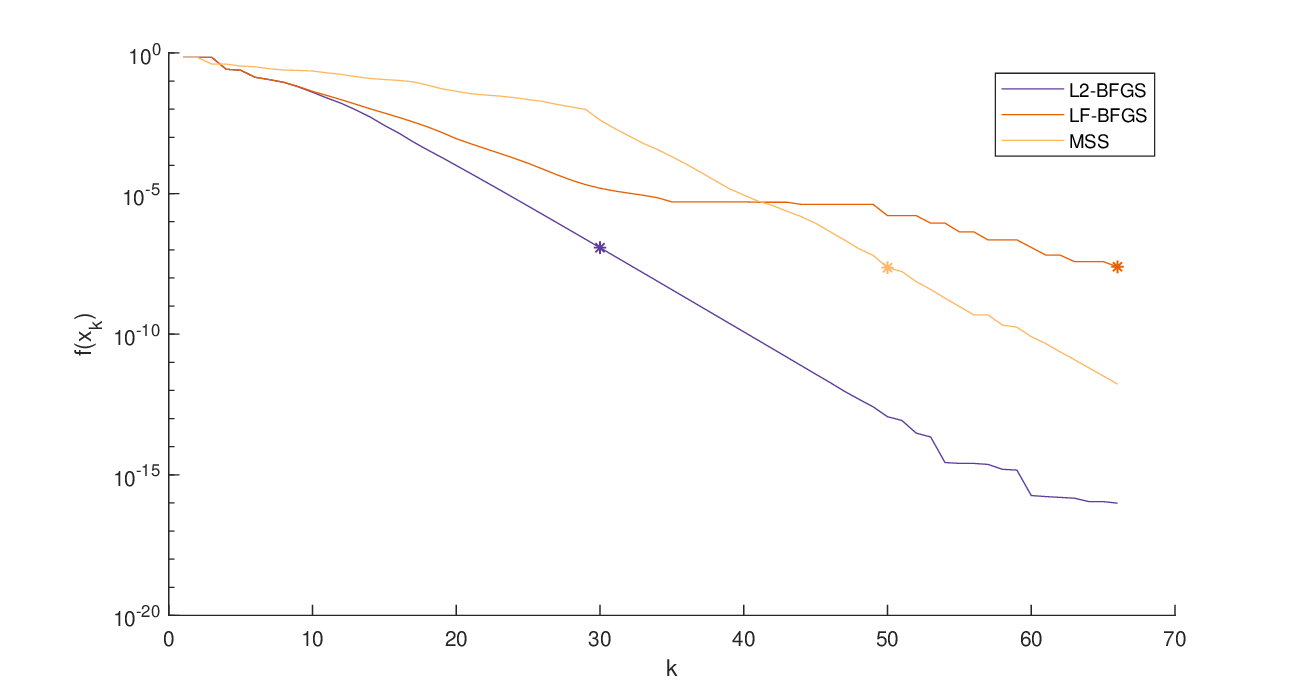}
    \caption{Results from the test with logistic regression, when using the best values of $m$ for each algorithm. The stars on the graphs mark at which point each of the algorithms fulfill the convergence condition $\|\nabla f(x_k)\|_2 \leq 10^{-6}$.}
    \label{fig:logistic_regression_gisette_comp}
\end{figure}

\subsection{Randomly generated quadratic problems}

In the setting of unconstrained optimization, the performance of an algorithm on quadratic problems can give a good indication of its convergence properties near a local optimum for general problems with analytic functions, as the quadratic term in the Taylor expansion dominates all other terms there. Quadratic problems also provide a convenient standard form for random problem generation. This suits the purpose of our second numerical test, which is a systematic comparison of the performance of MSS, L2-BFGS and LF-BFGS when applied to solve such problems, which naturally appear in applications in the form of linear least-squares problems.

For this test, we perform Monte Carlo simulations where 1000 problems with objective function on the form $\Vert Cx-d\Vert^2$ with $x\in \mathbb{R}^{1000}$ are generated. The left and right singular vectors of the matrices $C$ are generated in the same way as in \cite{Random_QP}, while the eigenvalues are sampled from a log-normal distribution where the logarithms have mean 0 and variance 1. The vector $d$ is chosen as $C\mathbbm{1}$. All iterations start from $x_0 = \mathbf{0}$ and for each iteration, the distance to the known optimal solution is recorded. The results of the Monte Carlo simulation can be seen in Figure \ref{fig:Monte_Carlo_QP_lognormal_m5}, for a test when $m=5$ and Figure \ref{fig:Monte_Carlo_QP_lognormal_m16}, for a similar test with $m=16$. The plots show the mean 10-logarithms of the normalized Euclidean distances to the optimum, plotted against the iteration number and with 3 standard deviation confidence interval borders in dashed lines. These confidence intervals do not overlap for most iterations, indicating a statistically significant difference in the algorithms' performances. In both tests, the mean distances to the optimal solutions are the lowest for LF-BFGS for most of the iterations. This corroborates the hypothesis put forth in Section \ref{sec:limitations} that reducing with respect to the Frobenius norm is better when the Hessian of the objective function does not change between iterations. A comparison of Figure \ref{fig:Monte_Carlo_QP_lognormal_m5} and Figure \ref{fig:Monte_Carlo_QP_lognormal_m16} reveals that while both LF-BFGS and MSS get an improved performance from the increase in $m$, L2-BFGS does not. 

\begin{figure}
\begin{subfigure}{.47\textwidth}
  \centering
  \includegraphics[width=\linewidth,height=.7\linewidth]{{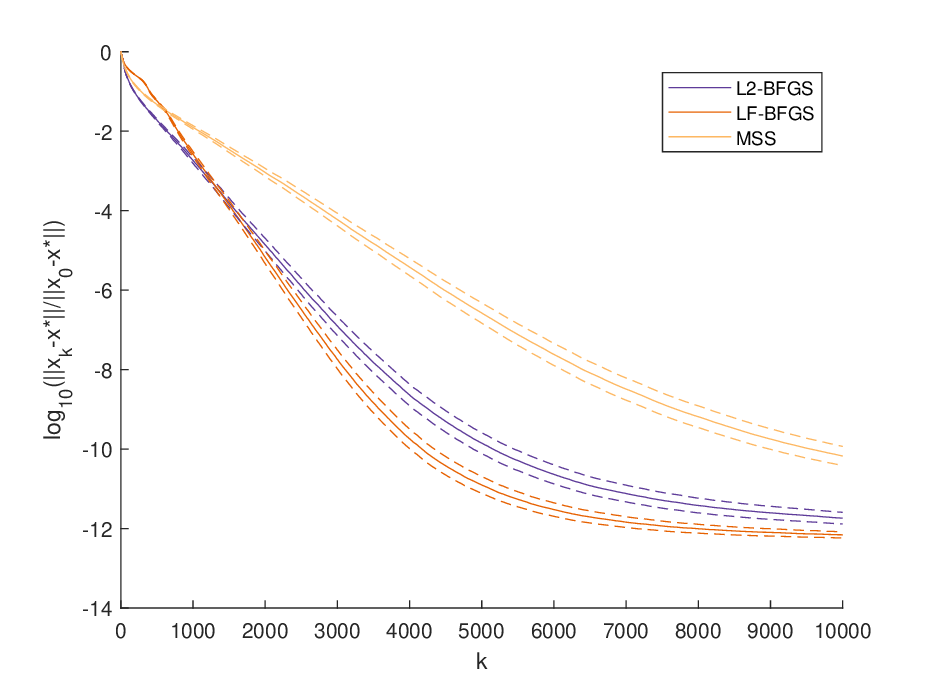}}
  \caption{m=5}
  \label{fig:Monte_Carlo_QP_lognormal_m5}
\end{subfigure}
\begin{subfigure}{.47\textwidth}
  \centering
  \includegraphics[width=\linewidth,height=.7\linewidth]{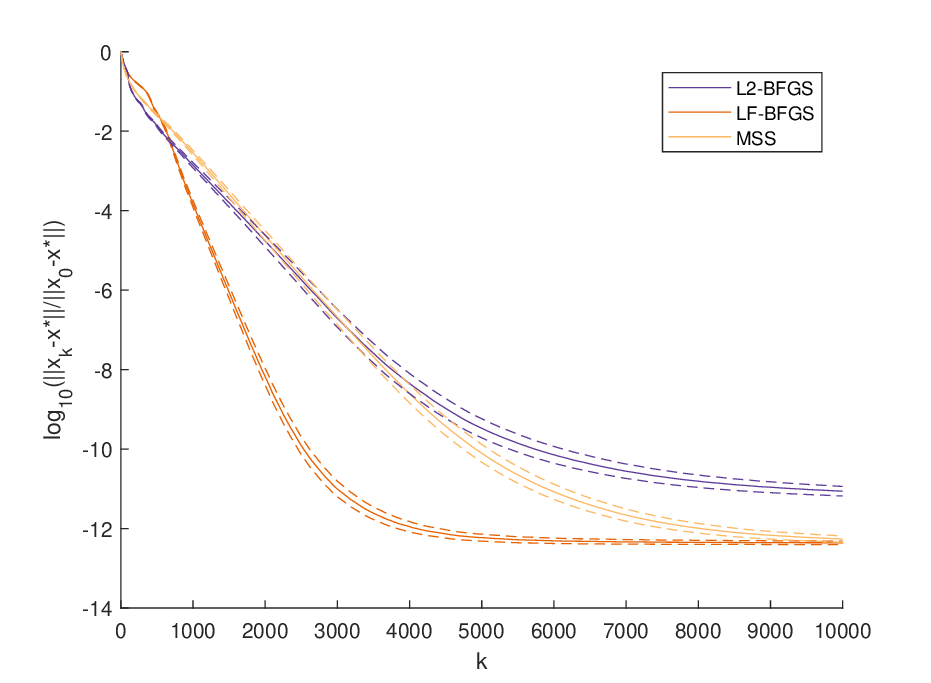}
  \caption{m=16}
  \label{fig:Monte_Carlo_QP_lognormal_m16}
\end{subfigure}
\caption{Results for the randomly generated QPs. The average 10-logarithm of the normalized Euclidean distance to the optimum as a function of iteration number k, for each of the three algorithms, with $\pm 3 \sigma$ confidence intervals.}
\label{fig:Agg_consecutive}
\end{figure}

\subsection{CUTEst test problems}

The final test consists of a battery of test-problems for unconstrained optimization from the CUTEst test suite \cite{gould2015cutest}, each defined by an objective function to be minimized and a starting point. The purpose of this test is to run our algorithms on a well-known collection of benchmark problems, and to compare our results with a test with MSS performed in a previous study. We replicate the tests performed in \cite{MSS} with our algorithms L2-BFGS and LF-BFGS, and count the number of function evaluations needed to solve each problem. To get consistency with the results reported in~\cite{MSS}, we use the same termination criterion as in that study: a problem is considered "solved" whenever the norm of its gradient falls below either $10^{-5}$, $10^{-6}$ times its initial value, or $10^{-6}$ times the initial function value. For further benchmarking, we take note of that L-BFGS is more commonly implemented as a line-search algorithm, and implement the conventional L-BFGS-method using the matrix representations detailed in \cite{Byrd_representations}, combining it with a line search algorithm for satisfying the strong Wolfe-conditions. By considering the performance in terms of function evaluations, we obtain a fair comparison between trust-region algorithms, which require one function evaluation per iteration, and line search algorithms, which usually require several. The selected values of $m$ were $5$ for L-BFGS with line search, $2.5$ (i.e. 5 stored eigenvectors) for L2-BFGS and LF-BFGS. A comparison is also made with the results of the test runs of MSS in \cite{MSS}, which used $m=5$.
\begin{table}
    \centering
    \begin{tabular}{|c|c|c|c|c|} \hline
      Problem name  &LF-BFGS & L2-BFGS & MSS & L-BFGS (Line search) \\ \hline
         ARWHEAD & \bf{12} & \bf{12} &15 & 35\\ \hline
         BDQRTIC &111 &78 & \bf{40} & 133\\ \hline
         BROYDN7D &\bf{940} &987 &1566 & 5000*\\ \hline
         BRYBND &36 &\bf{34} &59 & 131\\ \hline
         CHAINWOO &41 &\bf{40} &54 & 58\\ \hline
         COSINE &430 &64 & \bf{14} & 29\\ \hline
         CRAGGLVY &123 &72 & \bf{36} & 74\\ \hline
         DIXMAANA &22 &22 & \bf{13} & 22\\ \hline
         DIXMAANB &17 &17 &\bf{13} & 15\\ \hline
         DIXMAANC &21 &21 & \bf{14} & 21\\ \hline
         DIXMAAND &26 &26 & \bf{15} & 16\\ \hline
         DIXMAANE &258 &80 & \bf{54} & 73\\ \hline
         DIXMAANF &108 &60 & \bf{24} & 49\\ \hline
         DIXMAANG &66 &57 & \bf{21} & 39\\ \hline
         DIXMAANH &46 &46 & \bf{20} & 27\\ \hline
         DIXMAANI &388 &114 & \bf{84} & 125\\ \hline
         DIXMAANJ &125 &65 & \bf{28} & 53\\ \hline
         DIXMAANK &94 &62 & \bf{24} & 44\\ \hline
         DIXMAANL &48 &48 & \bf{22} & 29\\ \hline
         DQDRTIC &11 & 11 &11 & \bf{7} \\ \hline
         DQRTIC &\bf{20} &\bf{20} &29 & 21\\ \hline
         EDENSCH &40 &40 & \bf{22} & 26\\ \hline
         EG2 &\bf{4} &\bf{4} &5 & 36\\ \hline
         ENGVAL1 &23 &23 & \bf{17} & 25\\ \hline
         EXTROSNB &89 &70 & \bf{39} & 114\\ \hline
         FMINSRF2 &553 &\bf{463} &569 & 457\\ \hline
         FMINSURF &913 &668 &\bf{241} & 804\\ \hline
    \end{tabular}
    \caption{Number of function evaluations for each of the compared algorithms. MSS function evaluations taken from \cite{MSS}.}
    \label{tab:fcn_table}
\end{table}

\begin{table}
    \centering
    \begin{tabular}{|c|c|c|c|c|} \hline
     Problem name  &LF-BFGS & L2-BFGS & MSS & L-BFGS (Line search) \\ \hline
        FREUROTH &701 &61 & \bf{36} & 52\\ \hline
         LIARWHD &\bf{24} &\bf{24} &30 & 60\\ \hline
         MOREBV &\bf{65} &69 &261 & 88\\ \hline
         NCB20 &933 & \bf{233} &772 & 1324\\ \hline
         NCB20B &56 & \bf{37} &49 & 86\\ \hline
         NONCVXU2 &19 &19 &19 & \bf{5}\\ \hline
         NONCVXUN &19 &19 &19 & \bf{7}\\ \hline
         NONDIA &\bf{3} &\bf{3} &4 & 17\\ \hline
         NONDQUAR &1567 &125 & \bf{47} & 85 \\ \hline
         PENALTY1 &\bf{19} & \bf{19} &25 & 22\\ \hline
         POWELLSG &36 &36 & \bf{29} & 46\\ \hline
         POWER &103 &97 & \bf{37} & 47\\ \hline
         QUARTC &\bf{20} &\bf{20} &29 & 21\\ \hline
         SCHMVETT &100 &74 & \bf{46} & 69\\ \hline
         SINQUAD &39 &39 & \bf{36} & 106\\ \hline
         SPARSINE &125 &122 & \bf{115} & 1629\\ \hline
         SPARSQUR &57 &57 & \bf{17} & 29\\ \hline
         SPMSRTLS &885 &135 & \bf{114} & 211\\ \hline
         SROSENBR & \bf{20} & \bf{20} &22 & 51\\ \hline
         TESTQUAD &107 &84 & \bf{72} & 142\\ \hline
         TOINTGSS &12 &12 &11 & \bf{3}\\ \hline
         TQUARTIC &\bf{40} &\bf{40} &47 & 82\\ \hline
         TRIDIA &817 &269 &\bf{263} & 487\\ \hline
         VARDIM &17 &17 & \bf{12} & 43\\ \hline
         VAREIGVL &49 &43 & \bf{31} & 32\\ \hline
         WOODS &33 &33 & \bf{22} & 44\\ \hline
    \end{tabular}
    \caption{Number of function evaluations for each of the compared algorithms. MSS function evaluations taken from \cite{MSS}.}
    \label{tab:fcn_table_2}
\end{table}
         
As seen in tables \ref{tab:fcn_table} and \ref{tab:fcn_table_2}, none of the compared algorithms has a unilaterally superior performance across all the test problems. Figure~\ref{fig:performance_profile} shows a performance profile for function evaluations, similar to the one found in~\cite{MSS}. For the set ${\mathcal P}$ of test problems under consideration, $r_{s,p}$ is the ratio of the function evaluations required to solve problem $p$ with method $s$ and the minimal number of function evaluations required to solve problem $p$ using any of the methods of consideration. Defining
\begin{align*}
    \pi_s(\tau) = \dfrac{1}{\text{card}(\mathcal{P})}\text{card}(\{p\in \mathcal{P}:\log_2(r_{s,p}) \leq \tau\}),
\end{align*}
\begin{figure}
    \centering
    \includegraphics[width=0.6\textwidth]{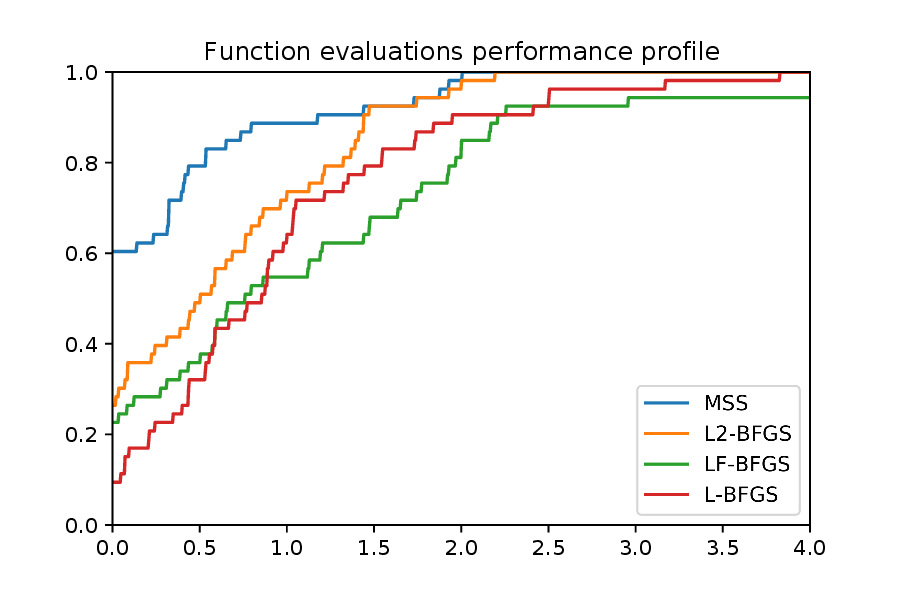}
    \caption{Performance profile for the L2-BFGS, LF-BFGS, MSS and L-BFGS methods.}
    \label{fig:performance_profile}
\end{figure}
Figure \ref{fig:performance_profile} plots $\pi_s(\tau)$ for $\tau \in [0,4]$. Although L2-BFGS has the best performance in $26.4\%$ of the problem instances in the test batch, the performance profile indicates that MSS seemed to perform the best overall. In particular, the other algorithms struggled with the DIXMAANA-DIXMAANL problems in comparison. However, the figure also shows that L2-BFGS had an overall better performance that L-BFGS with line search, while the performance of LF-BFGS was comparable. 

Since all algorithms tested were based on the BFGS update, their difference in performance was caused by other aspects of the algorithm. Worth noting is that both MSS and L-BFGS with line search used the strategy of dropping old curvature information to handle the memory limit, but obtained widely different results. All algorithms also enforced positive definiteness of their Hessian approximations: The line-search algorithm by ensuring that the Wolfe condition was satisfied, and the trust-region algorithms by skipping updates when $y_k^T s_k < \varepsilon \|s_k|_2 \|y_k\|_2$. 

Besides the number of function evaluations, one can consider the runtime of the compared algorithms. In this regard, we note that the CUTEst data set is not the ideal battery test to evaluate this for LF-BFGS and L2-BFGS, as function and gradient evaluations are not the main bottleneck when solving these optimization problems. The line search implementation of L-BFGS, with matrix representations as detailed in \cite{Byrd_representations} is made to be very efficient per iteration, avoiding as many operations with $n$-dimensional vectors as possible. Our implementations of L2-BFGS and L2-BFGS perform more such computations, and the test batch does not have the long function evaluation times that allow us to perform the matrix reduction in parallel while otherwise waiting for the gradient computations to return. As a result, their runtimes for the CUTEst problems are sometimes significantly longer. Thus, in practice, although LF-BFGS was shown to outperform L-BFGS with line search in terms of function evaluations, its runtime would only be shorter than L-BFGS for problems where function evaluations are the main bottleneck, such as when the objective function is defined by large amounts of data.

\section{Conclusions and further work}

In this paper, we have derived analytical solutions to matrix optimization problems where the objective function is either a unitarily invariant matrix norm, or based on the Stein-loss. From these results, we have derived new limited-memory versions of quasi-Newton methods. These limited memory algorithms make use of a low-rank plus shift structure of their Hessian approximations, which we characterize as a constraint on the eigenvalues of the matrix. We outlined how to leverage this matrix structure to construct efficient trust-region methods and demonstrated competitive performance for trust-region implementations of our type of limited memory adaption of BFGS, compared to MSS and conventional L-BFGS. 

While we took a broad view when deriving solutions to matrix optimization problems, and showed how they could be used to derive limited memory versions of every quasi-Newton method in the Broyden class, our numerical evaluations were limited in scope. Therefore, there is much room for further development. A more comprehensive investigation of which matrix optimization problems and which Broyden-class methods would combine well together is warranted. In particular, we suspect that a trust-region implementation of SR1 with matrix reduction based on the 2-norm could outperform the methods implemented here for the non-convex problems in the CUTEst collection. Another direction for future work is to further investigate different ways of handle the memory limits of the method. This includes a more systematic investigation of how the choice of memory parameter affects performance, as well as whether performing matrix reductions with a different frequency than once per iteration could yield better results. Finally, it would be interesting to look into how these quasi-Newton methods would perform in a stochastic setting, as their way of aggregating curvature information could prove advantageous there.

\bibliographystyle{ieeetr}
\renewcommand{\bibname}{References}
\bibliography{main} 
\appendix

\section{Appendix}

\subsection{Results on matrix distance measures}\label{sec:matrixdistance}

This section presents results used to derive the theorems in Section~\ref{sec:NLMMP}
\begin{theorem}[{{\cite[page 467]{matrix}}}]\label{thm:Unitary matrix norms}
Let $m$ and $n$ be arbitrary positive integers. Let $A$ and $B$ be arbitrary $m \times n$ matrices, and let, let $V_1 \Sigma (A) W_1^*$ and $V_2 \Sigma(B) W_2$ be singular value decompositions of them. Then, for any unitarily invariant norm $\|\cdot\|$ on the set of $m \times n$ matrices, $\|A - B\| \geq \|\Sigma(A) - \Sigma(B)\|$.
\end{theorem}
Theorem \ref{thm:Unitary matrix norms} has special implications for Hermitian matrices, for which the singular values are the absolute values of the eigenvalues. Corollary \ref{cor:Mirsky} states those implications.
\begin{corollary}[{{\cite[page 468]{matrix}}}]\label{cor:Mirsky}
    Let $A,B$ be Hermitian $n \times n$ matrices and let $\|\cdot\|$ be a unitarily invariant norm on the set of $n \times n$ matrices. Then, 
    \begin{equation*}
        \|\diag (\lambda^{\downarrow}(A)) - \diag (\lambda^{\downarrow}(B)) \| \leq \|A - B\|.
    \end{equation*}
\end{corollary}
From Corollary \ref{cor:Mirsky}, an important lemma for solving certain nearest limited memory matrix problems will be derived.

\begin{lemma}\label{lemma:ABC}
    For all Hermitian matrices $A$ and $B$ of the same size and any unitarily invariant matrix norm $\|\cdot\|$, there is a Hermitian matrix $C$ with the same eigenvectors as $A$ and the same eigenvalues as $B$ such that $\|A - C\| \leq \|A - B\|$.  
\end{lemma}

\begin{proof}
    Since $A$ and $B$ are Hermitian, they can be eigendecomposed as $A = U\diag( \lambda^{\downarrow}(A)) U^* $ and $B = V \diag (\lambda^{\downarrow}(B)) V^*$, where $U$ and $V$ are unitary matrices. Setting $C = U \diag (\lambda^{\downarrow}(B)) U^*$, and making use of unitary invariance of $\|\cdot \|$,
    \begin{equation*}
    \begin{aligned}
        \|A - C\| = \|U \left (\diag (\lambda^{\downarrow}(A)) - \diag (\lambda^{\downarrow}(B)) \right) U^* \| = \|\diag (\lambda^{\downarrow}(A)) - \diag (\lambda^{\downarrow}(B))\|.
    \end{aligned}
    \end{equation*}
    By Corollary \ref{cor:Mirsky}, it follows that $\|A-C\| \leq \|A - B\|$. 
\end{proof}

Another useful result for solving certain nearest limited memory matrix problems is Lemma~\ref{lemma:logdetineq}, which can be derived from the following theorem:

\begin{theorem} [{{\cite[page 255]{matrix}}}]\label{thm:trace_ineq}
    Let $A,B$ be Hermitian $n \times n$ matrices. Then,
    \begin{equation*}
        \sum_{i=1}^n \lambda_i(A) \lambda_{n+1-i}(B) \leq \tr(AB) \leq \sum_{i=1}^n \lambda_i(A) \lambda_{i}(B).
    \end{equation*}
\end{theorem}

\begin{lemma}\label{lemma:logdetineq}
    For all Hermitian matrices $A$ and $B$ of the same size, there is a Hermitian matrix $C$ with the same eigenvectors as $A$ and the same eigenvalues as $B$ such that 
    \begin{align*}
     \tr(B^{-1}A) - \log \det (B^{-1}A) &\geq \tr (C^{-1} A) - \log \det (C^{-1}A),\\
     \tr(BA^{-1}) - \log \det (BA^{-1}) &\geq \tr (C A^{-1}) - \log \det (CA^{-1}), \mbox{ and }\\
      \tr(B^{-1}A) + \tr(BA^{-1})  &\geq \tr (C^{-1} A) + \tr (CA^{-1}),
    \end{align*}
\end{lemma}
\begin{proof}
    Since the determinant of a product of two matrices is equal to the product of eigenvalues of those matrices, and since $B$ and $C$ have the same eigenvalues, $\log\det(B^{-1}A) = \log \det(C^{-1}A)$ and $\log\det(BA^{-1}) = \log \det(CA^{-1})$, so we focus on the trace terms. Let us begin by showing that $\tr(B^{-1}A) \geq \tr(C^{-1}A).$ Since $A$ is Hermitian, it can be eigendecomposed as $U\diag(\lambda^{\downarrow}(A))U^*$. Thus, setting $C = U \diag(\lambda^{\downarrow}(B))U^*$ gives that
    \begin{equation*}
    \begin{aligned}
        & \tr(C^{-1}A) = \tr(U\diag(\lambda^{\downarrow}(B))^{-1}\diag(\lambda^{\downarrow}(A))U^*) \\
        & = \sum_{i=1}^n \lambda_{i}(B^{-1}) \lambda_{n+1-i}(A) \leq \tr(B^{-1}A).
    \end{aligned}
    \end{equation*}
    where the last inequality follows from Theorem~\ref{thm:trace_ineq}. Similarly,
    \begin{equation*}
    \begin{aligned}
        & \tr(CA^{-1}) = \tr(U\diag(\lambda^{\downarrow}(B))\diag(\lambda^{\downarrow}(A))^{-1}U^*) \\
        & = \sum_{i=1}^n \lambda_{i}(B) \lambda_{n+1-i}(A^{-1}) \leq \tr(BA^{-1}).
    \end{aligned}
    \end{equation*}
    and therefore
    \begin{equation*}
    \begin{aligned}
        & \tr(C^{-1}A) + \tr(CA^{-1}) \leq  \tr(B^{-1}A) + \tr(BA^{-1}).
    \end{aligned}
    \end{equation*}
    Combining these results leads to the desired conclusion.
\end{proof}
Note that if $A$ is a real symmetric matrix, the matrix $C$ as constructed in the proofs of lemmas \ref{lemma:ABC} and \ref{lemma:logdetineq} will be real and symmetric. 

\subsection{Eigenvalue decomposition of limited memory matrices}\label{sec:Eigendecomposition}

To obtain orthonormal eigenvectors for 
\begin{equation}\label{eqn:B_hat}
    \widehat B_{k+1} = \alpha_k + \widehat U_{k+1} \widehat C_{k+1} \widehat U_{k+1}^T,
\end{equation}
we utilize QR-decomposition. Similarly to what was done in \cite{Combining_TR_and_LQN}, We consider the Gram matrix $G_{k+1} := \widehat U_{k+1}^T \widehat U_{k+1}.$ An LDLT-factorization of $G_{k+1}$ with pivoting yields a permutation matrix $P_{k+1}$, a lower triangular matrix $L_{k+1}$ and a diagonal matrix $D_{k+1}$ such that $P_{k+1}^T G_{k+1} P_{k+1} = L_{k+1} D_{k+1} L_{k+1}^T$. 
With the QR-decomposition ansatz $ \widehat U_{k+1} P_{k+1} = Q_{k+1} R_{k+1}$, $P_{k+1}^T G_{k+1} P_{k+1} = R_{k+1}^T R_{k+1} = L_{k+1} D_{k+1} L_{k+1}^T$, which implies that $R_k=\sqrt{D_k}L_k^T$. Consider first the case when $R_{k+1}$ has full rank. The QR decomposition is then unique and $Q_{k+1}$ can be computed as $Q_{k+1} = \widehat U_{k+1} P_{k+1} L_{k+1}^{-T} \sqrt{D_{k+1}}^{-1}$. With $Q_{k+1}$ computed, consider Equation~\eqref{eqn:B_hat} again. Inserting $\widehat U_{k+1} = Q_{k+1} \sqrt{D_{k+1}} L_{k+1}^T P_{k+1}^T$ gives 
\begin{equation}\label{eqn:update_LDP}
    \widehat B_{k+1} = \alpha_k I + Q_{k+1} \sqrt{D_{k+1}} L_{k+1}^T P_{k+1}^T \widehat C_{k+1} P_{k+1} L_{k+1} \sqrt{D_{k+1}} Q_{k+1}^T. 
\end{equation}
Performing an eigenvalue decomposition of \\ $\sqrt{D_{k+1}} L_{k+1}^T P_{k+1}^T \widehat C_{k+1} P_{k+1} L_{k+1} \sqrt{D_{k+1}}$ yields 
\begin{equation*}
    \sqrt{D_{k+1}} L_{k+1}^T P_{k+1}^T \widehat C_{k+1} P_{k+1} L_{k+1} \sqrt{D_{k+1}} = V_{k+1} \Lambda_{k+1} V_{k+1}^T,
\end{equation*}
where $V_{k+1}$ is an orthonormal and $\Lambda_{k+1}$ a diagonal matrix. Since $Q_{k+1}$ is also an orthonormal matrix, $ E_{k+1} = V_{k+1} Q_{k+1}$ is orthonormal, and the columns of $ E_{k+1}$ can be identified as eigenvectors of $\widehat B_{k+1}$, while the diagonal elements of $\Lambda_{k+1}$ can be identified as the eigenvalues of $\widehat B_{k+1}$ subtracted by $\alpha_k$. This can be seen from
\begin{equation}\label{eqn:new_eigen}
    \widehat B_{k+1} = \alpha_k I + E_{k+1} \Lambda_{k+1} E_{k+1}^T.
\end{equation}
 Even though $G_{k+1}$ should be positive semidefinite, errors due to finite precision arithmetic may cause the diagonal elements of $D_{k+1}$ to be negative. Also, in the case where $R_{k+1}$ does not have full rank in exact arithmetic, numerical rounding errors may still render its diagonal elements nonzero.
 To detect and handle these situations, the following scheme is implemented: If any of the diagonal elements of $D_{k+1}$ are below a given threshold $\nu$, they are considered to be zero. Let $D_{k+1}^{\dagger}$ be the matrix obtained by deleting the rows and columns of $D_{k+1}$ where there are diagonal elements below the threshold. From $Q_{k+1} \sqrt{D_{k+1}} = U_{k+1} P_{k+1} L_{k+1}^{-T}$, it can be seen that the range of $U_{k+1}$ is spanned by those columns of $Q_{k+1}$ that correspond to nonzero diagonal elements of $D_{k+1}$. Therefore, only those columns need to be computed. This can be done by first forming the matrix $U_{k+1} P_{k+1} L_{k+1}^{-T}$, then deleting the columns of this matrix corresponding to zero diagonal elements of $D_{k+1}$, and lastly multiplying the resulting matrix from the right by $( D_{k+1}^{\dagger})^{-1/2}$. The resulting matrix will be denoted $Q_{k+1}^{\dagger}$. To derive the new eigendecomposition, let $X_{k+1}$ be the matrix obtained by taking the identity matrix and deleting rows corresponding to zero diagonal elements of $D_{k+1}$. Also, let $L_{k+1}^{\dagger}$ be the matrix obtained by deleting columns in $L_{k+1}$ that would have been multiplied by the deleted diagonal elements of $D_{k+1}$ in the LDLT factorization. Then, the following identities hold:
\begin{equation}\label{eqn:removal_identities}
\begin{aligned}
    \sqrt{D_{k+1}^{\dagger}} &= X_{k+1} \sqrt{D_{k+1}} X_{k+1}^T,\\
    Q_{k+1}^{\dagger} &= Q_{k+1} X_{k+1}^T, \\
    \sqrt{D_{k+1}} &= X_{k+1}^T X_{k+1} \sqrt{D_{k+1}} = \sqrt{D_{k+1}} X_{k+1}^T X_{k+1}, \\
    L_{k+1}^{\dagger} &= L_{k+1} X_{k+1}^T.
\end{aligned}
\end{equation}
Using these identities in equation \eqref{eqn:update_LDP} gives
\begin{equation*}
\begin{aligned}
    \widehat B_{k+1} &= \alpha_k I + \\
    & Q_{k+1} X_{k+1}^T X_{k+1} \sqrt{D_{k+1}} X_{k+1}^T X_{k+1} L_{k+1}^T P_{k+1}^T \widehat C_{k+1} P_{k+1} \\
    &L_{k+1} X_{k+1}^T X_{k+1} \sqrt{D_{k+1}} X_{k+1}^T X_{k+1} Q_{k+1}^T \\
    &= \alpha_k I + Q_{k+1}^{\dagger} \sqrt{D_{k+1}^{\dagger}} {L_{k+1}^{\dagger}}^T P_{k+1}^T \widehat C_{k+1} P_{k+1} L_{k+1}^{\dagger} \sqrt{D_{k+1}^{\dagger}} {Q_{k+1}^{\dagger}}^T.
\end{aligned}
\end{equation*}
The computation of new eigenvectors and eigenvalues now proceeds as previously by computing an eigenvalue decomposition $V_{k+1} \Lambda_{k+1} V_k^T$ of \\ $\sqrt{D_{k+1}^{\dagger}} {L_{k+1}^{\dagger}}^T P_{k+1}^T \widehat C_{k+1} P_{k+1} L_{k+1}^{\dagger} \sqrt{D_{k+1}^{\dagger}}$ and identifying the columns of $Q^{\dagger}_{k+1} V_{k+1}$ as the new eigenvectors, resulting in an updated matrix on the form $\eqref{eqn:new_eigen}$. In total, this update has complexity $O(m^3 + nm^2)$.

\subsection{Trust region algorithms} \label{sec:TR}
Besides line search algorithms, another paradigm for controlling step length in iterative optimization is trust region algorithms. Instead of first determining the direction to move in, trust region algorithms begin by choosing a maximum step length and then determining the actual step by solving a model problem with this step length constraint. In what follows, a trust-region framework will be described. Given a problem on the form
\begin{equation*}
\begin{aligned}
& \underset{x \in \mathbb{R}^n}{\text{minimize}}
& & f(x), \\
\end{aligned}
\end{equation*}
where $f :\mathbb{R}^n \mapsto \mathbb{R}$ is twice continuously differentiable, the trust region algorithm solves a series of model problems on the form 
\begin{equation}\label{eqn:TRproblem}
\begin{aligned}
& \underset{p \in \mathbb{R}^n}{\text{minimize}}
& &g_k^Tp + \frac{1}{2}p^TB_{k}p \\
& \text{subject to}
& & \|p\|_2 \leq d_k, \\
\end{aligned}
\end{equation}
where $g_k$ is the gradient and $B_{k}$ the approximation of the Hessian of $f$ evaluated at the current candidate solution $x_k$. The solution $p_k$ to the trust region problem is then used to compute trial solution $x_{\rm trial} = x_k + p_k$. The trust region radius $d_k$ is increased, decreased or unchanged between iterations, depending on how well the predicted change in objective function agrees with the actual change $f(x_{\rm trial})-f(x_k)$, and $x_{\rm trial}$ is accepted as the new candidate solution if there is sufficient agreement. The algorithmic framework for a trust region algorithm, similar to the one presented by Sorensen \cite{trust_newton}, is summarized in Algorithm~\ref{alg:tr_alg_fr}.
\begin{algorithm}[H]\small
\caption{Basic framework for a trust region algorithm.}
\label{alg:tr_alg_fr}
\begin{algorithmic}
\item Specify constants $\eta_1, \eta_2, \gamma_1, \gamma_2$ such that $0< \eta_1< \eta_2<1$, $0<\gamma_1<1<\gamma_2$.
\State $k \gets 0$
\item Initialize $x_k$, $d_k$ 
\State $\text{converged} \gets \text{False}$
\While {not converged}
    \State $g_k \gets \nabla f(x_k)$
    \State $B_{k} \gets [\text{approximation of } \nabla^2 f(x_k)]$
    \State $p_k \gets [\text{solution of }$ \eqref{eqn:TRproblem}]
    \State $ared \gets f(x_k+p_k) - f(x_k)$
    \State $pred \gets g_k^Tp_k + \frac{1}{2}p_k^TB_{k}p_k$
    \If{$\dfrac{ared}{pred} < \eta_1$}
        \State $d_{k+1} \gets \gamma_1 d_k$
        \State $x_{k+1} \gets x_k$
    \ElsIf{$\dfrac{ared}{pred}> \eta_2$}
        \State $d_{k+1} \gets \gamma_2 d_k$
        \State $x_{k+1} \gets x_k + p_k$
    \Else 
        \State $x_{k+1} \gets x_k+p_k$
    \EndIf
    \State $\text{converged} \gets [\text{result of some convergence test}]$
    \State $k \gets k+1$
\EndWhile \\
\Return $x_k$
\end{algorithmic}
\end{algorithm}

To design a trust-region algorithm in the above framework, it is necessary to specify the parameters for trust region expansion and contraction, the method for approximating the Hessian, the solution method for the trust region problem \eqref{eqn:TRproblem}, and the convergence criterion. More general frameworks could also include different ways of updating the trust region and consider different choices of norm than the $l^2$ norm in \eqref{eqn:TRproblem}. One of the main ideas behind solution methods of \eqref{eqn:TRproblem} is the following theorem:
\begin{theorem}[{\cite[Theorem 2.1]{Locally_constrained_steps}}] \label{thm:trust_solution}
A vector $p$ is a global minimum of the trust region problem \eqref{eqn:TRproblem} if and only if $\|p\|_2 \leq d$ and there exists a unique $\sigma \geq 0$ such that 
\begin{equation}\label{eqn:tr_optimality}
    (B_{k}+\sigma I)p = -g_k, \quad \sigma(d_k-\|p\|_2)=0,
\end{equation}
with $B_{k} + \sigma I$ positive semidefinite. Moreover, if $B_{k}+\sigma I$ is positive definite, then the global minimizer is unique.
\end{theorem}
Solving the trust region problem is thus equivalent to finding $\sigma$ and $p$ that satisfy the conditions above. As $\sigma\geq 0$ and $\sigma I + B_k$ should be positive semidefinite, $\sigma \geq \max(0,-\lambda_1(B_{k}))$. The solution of the resulting system of equations and inequalities can be divided into cases based on whether the optimal $\sigma$ fulfills this inequality strictly or not.
\paragraph{Case 1, $\sigma > (\max(0,-\lambda_1(B_{k}))):$}
In this case, $B_{k}+\sigma I$ will be positive definite, the equality $\sigma(d_k-\|p\|_2) = 0$ can only be fulfilled when $\|p\|_2 = d_k$ and $p$ can be uniquely defined as a function of $\sigma$ by $p(\sigma)=-(B_{k}+\sigma I)^{-1}g_k$. Then, the problem is reduced to finding a $\sigma$ such that $\|p(\sigma)\|_2 = d_k$. According to a review of trust region methods by Moré \cite{trust_region_review}, a numerically suitable way of doing so is to apply Newton's method for zero-finding to the function $\psi(\sigma) = \frac{1}{d_k}-\frac{1}{\|p(\sigma)\|_2} $. The Newton iterations for finding an optimal $\sigma$ can be described by the formula: 
\begin{equation}\label{eqn:sigmaupdate}
\sigma_{new} = \sigma + \dfrac{\|p(\sigma)\|_2^2}{p(\sigma)^T(B_{k}+\sigma I)^{-1}p(\sigma)}\Big(\dfrac{\|p(\sigma)\|_2 - d_k}{d_k} \Big). 
\end{equation}
A direct computation of the matrix inverse in this formula can be computationally expensive. In \cite{Trust_region_step}, a Cholesky factorization of $(B_{k}+\sigma I)$ is used to avoid this, and any method utilizing this type of iteration should address this problem. Another issue to address is that with a finite number of iterations of \eqref{eqn:sigmaupdate}, $p(\sigma)$ is not guaranteed to lie in the trust region. In order to remedy this, a tolerance parameter $\varepsilon$ can be chosen to define a contracted trust region radius $\Delta_k = \frac{d_k}{1+\varepsilon}$. The iterations can then be stopped when the convergence criterion $\frac{|\|p\|_2-\Delta_k|}{\Delta_k} < \varepsilon$ is satisfied. Lemma~3.13 in~\cite{Trust_region_step} implies that if $Q$ is the objective value of \eqref{eqn:TRproblem} for a solution that satisfies this convergence criterion and $Q^*$ is the optimal value of this problem, then 
\begin{equation}\label{eqn:tr_opt_sol_result}
    Q^* \leq Q \leq (1-\varepsilon)^2 Q^*.
\end{equation}
To ensure that new iterates do not fall outside the interval $(\max(0,-\lambda_1(B_{k})),\infty)$, a safeguarding technique can be used \cite{Trust_region_step}, e.g. setting the new $\sigma$ to the mean of the previous one and $\max(0,-\lambda_1(B_{k}))$ instead of performing the update in equation \eqref{eqn:sigmaupdate} when that happens.
\paragraph{Case 2: $\sigma = 0 < \lambda_1$:} In this case, $B_{k}$ must be positive definite, and the unique $p$ such that $B_{k}p = -g_k$, i.e. the solution of the unconstrained problem corresponding to \eqref{eqn:TRproblem}, has to lie inside the trust region. Detecting if this has occurred can be done by checking whether $\|B_k^{-1}g_k\|_2 \leq d_k$. 
\paragraph{Case 3: $\sigma = -\lambda_1 \geq 0$:}The case where $\sigma = - \lambda_1(B_{k})$ and $\lambda_1(B_{k}) \leq 0$ is referred to as the ``hard case'' in the literature. Solutions to $(B_{k} - \lambda_1(B_{k})I)p = -g_k$ are not unique, but one solution is $p^* = - (B_{k}-\lambda_1(B_{k})I)^{\dagger} g_k$, where $(B_{k} -\lambda_1(B_{k}) I)^{\dagger}$ denotes the Moore-Penrose pseudoinverse of $B_{k} -\lambda_1(B_{k}) I$. The hard case occurs if and only if $\|p^*\|_2 \leq d_k$. A solution to the trust region problem in this case is given by $p_k = p^* + \tau z$, where $z$ is any eigenvector in the eigenspace of $\lambda_1(B_{k})$ and $\tau = \frac{\sqrt{(\|d_k\|_2^2-\|p^*\|_2^2)}}{\|z\|_2}$. 

Algorithm~\ref{alg:tr_sol} summarizes the solution framework resulting from the above discussion.

\begin{algorithm}[H]\small
\caption{Framework for a solution algorithm to the trust region problem}
\label{alg:tr_sol}
\begin{algorithmic}
\item Specify $B_k$, $g_k$, $d_k$, $\varepsilon$
\If{$\lambda_1(B_k) > 0$} 
    \If{$\|B_k^{-1}g_k\|_2 \leq d_k $} 
        \State $p \gets -B_k^{-1}g_k$
    \Else 
        \State Initialize $\sigma > 0$
        \State $\Delta_k \gets \frac{d_k}{1+\varepsilon}$
        \While{$\frac{|\|p\|_2-\Delta_k|}{\Delta_k}>\varepsilon$}
            \State Compute $\|p\|_2$ and $p^T(B_k+\sigma I)^{-1}p$, where $ p = -(B_k + \sigma I)^{-1}g_k$
            \State $\sigma_{\rm trial} \gets \sigma + \dfrac{\|p\|_2^2}{p^T(B_k+\sigma I)^{-1}p}\Big(\dfrac{\|p\|_2 - \Delta_k}{\Delta_k} \Big).$ 
            \If{$\sigma_{\rm trial} \leq 0$}
                \State $\sigma \gets \frac{\sigma}{2}$
            \Else
                \State $\sigma \gets \sigma_{\rm trial}$ 
            \EndIf
        \EndWhile
        \State $p \gets -(B_k + \sigma I)^{-1}g_k$
    \EndIf
\Else
    \If{$\|(B_k-\lambda_1(B_k))^{\dagger}g_k\|_2 \leq d_k$}
        \State $p^* \gets -(B_k - \lambda_1(B_k)I)^{\dagger}g_k$
        \State $z \gets [ \text{Some eigenvector in the eigenspace of } \lambda_1(B_k) ]$
        \State $\tau \gets \frac{\sqrt{d_k^2 - \|p\|_2^2}}{\|z\|_2}$
        \State $p \gets p^* + \tau z$
    \Else 
        \State Initialize $\sigma > - \lambda_1(B_k)$
        \State $\Delta_k \gets \frac{d_k}{1+\varepsilon}$
        \While{$\frac{|\|p\|_2-\Delta_k|}{\Delta_k}>\varepsilon$}
            \State Compute $\|p\|_2$ and $p^T(B_k+\sigma I)^{-1}p$, where $ p = -(B_k + \sigma I)^{-1}g_k$
            \State $\sigma_{\rm trial} \gets \sigma + \dfrac{\|p\|_2^2}{p^T(B_k+\sigma I)^{-1}p}\Big(\dfrac{\|p\|_2 - \Delta_k}{\Delta_k} \Big).$ 
            \If{$\sigma_{\rm trial} \leq -\lambda_1(B_k)$}
                \State $\sigma \gets \frac{\sigma-\lambda_1(B_k)}{2}$
            \Else
                \State $\sigma \gets \sigma_{trial}$ 
            \EndIf
        \EndWhile
        \State $p \gets -(B_k + \sigma I)^{-1}g_k$
    \EndIf
\EndIf \\
\Return p
\end{algorithmic}
\end{algorithm}
We will now state a general convergence result for trust-region methods, valid not only for model problems with quadratic objective functions and trust regions limited by the $2$-norm. The theorem allows for the norm used to limit the trust region to change between iterations, denoting the norm used in iteration $k$ by $\|\cdot \|_k$. To distinguish this norm in iteration $2$ from the $2$-norm, we will denote the latter by $\|\cdot\|$ when stating the theorem.
\begin{theorem}[{{\cite[page 137]{Trust_region_methods}}}]\label{thm:trust-convergence}
Assume that the problem 
\begin{equation*}
\begin{aligned}
& \underset{x \in \mathbb{R}^n}{\text{minimize}}
& & f(x) \\
\end{aligned}
\end{equation*}
is solved by an iterative trust region algorithm, where the candidate solution is updated as $x_{k+1} = x_k + p_k$, and $p_k$ is obtained as the solution to the following trust region subproblem
\begin{equation*}
\begin{aligned}
& \underset{p_k \in \mathbb{R}^n}{\text{minimize}}
& & m_k(x_k+ p_k) \\
& \text{subject to}
& & \|p\|_k \leq d_k. \\
\end{aligned}
\end{equation*}
Here, $m_k(x)$ is a model function of $f(x)$ and $\|\cdot \|_k$ is a norm used to define the trust region in iteration $k$. Denote this trust region by $\mathcal{B}_k$ and assume that the following conditions are satisfied:
\begin{itemize}
    \item[(a)] The objective function $f(x)$ is twice continuously differentiable on $\mathbb{R}^n$
    \item[(b)] There exists a constant $\kappa_{lbf}$ such that $f(x) \geq \kappa_{lbf} \forall x\in\mathbbm{R}^n$.
    \item[(c)] There exists a positive constant $\kappa_{ufh}$ such that $\|\nabla^2f(x)\| \leq \kappa_{ufh} \forall x \in \mathbbm{R}^n$.
    \item[(d)] For all $k$, $m_k$ is twice differentiable.
    \item[(e)] For all $k$, $m_k(x_k) = f(x_k)$.
    \item[(f)] For all $k$, $g_k:= \nabla m_k(x_k) = \nabla f(x_k)$.
    \item[(g)] The Hessian of the model is uniformly bounded within the trust region over all iterations, \emph{i.e.} $\|\nabla^2 m_k(x)\| \leq \kappa_{umh} -1 \ \forall x \in \mathcal{B}_k,\ \forall k$, for some constant $\kappa_{umh}\geq 1$ independent of $k$.
    \item[(h)] There exists a constant $\kappa_{une} \geq 1$ such that for all $k$, $\frac{1}{\kappa_{une}}\|x\|_k \leq \|x\| \leq \kappa_{une}\|x\|_k \ \forall x \in \mathbbm{R}^n$.
    \item[(i)] Let $\beta_k = 1 + \max_{x \in \mathcal{B}_k }\| \nabla^2 m_k(x) \|$. Then, for all $k$, it holds that $m_k(x_k) - m_k(x_k + p_k) \geq \\ \kappa_{mdc}\|g_k\| \min \left [ 
    \frac{\|g_k\|}{\beta_k},d_k\right]$ with $\kappa_{mdc} \in (0,1)$.
\end{itemize}
Then, the trust region algorithm will converge to a critical point of $f(x)$.
\end{theorem}

An attentive reader might note that the trust-region problem \eqref{eqn:TRproblem} does not have an objective function that directly fulfills the assumptions of Theorem~\ref{thm:trust-convergence}. However, the problem is equivalent to a trust region problem that uses such a model function. To see this, consider $m_k(x) = f(x_k) + g_k^T(x-x_k) + \frac{1}{2}(x - x_k)^T B_{k} (x - x_k)$. The trust region problem \eqref{eqn:TRproblem} is obtained with the change of variables $p = x - x_k$ and by ignoring the constant term $f(x_k)$, since doing so does not affect the optimization.
 
The last assumption of Theorem \ref{thm:trust-convergence} is a sufficient decrease assumption, and can be proven to hold for trust region algorithms that utilize quadratic model functions by considering the Cauchy point $x_k^C$, which is the minimizer of $m_k(x)$ under the constraint that $x$ lies on the arc $\{x:x = x - t g_k, t \geq 0, x \in \mathcal{B}_k\}$. In \cite{Trust_region_methods}, the following theorem about the Cauchy point is proven, and utilized to show that the sufficient decrease condition holds for quadratic model functions:
\begin{theorem}[{{\cite[page 127]{Trust_region_methods}}}] \label{thm:C_point}
With a quadratic model function $m_k(x) = f(x_k) + g_k^T(x-x_k) + \frac{1}{2}(x - x_k)^TB_k(x - x_k)$, the Cauchy point fulfills the inequality 
\begin{equation*}
    m_k(x_k) - m_k(x_k^C) \geq \frac{1}{2} \|g_k\|^2 \min \left\{ \frac{1}{\beta_k},\frac{d_k}{\|g_k\|_k}\right\}.
\end{equation*}
\end{theorem}
In order to use the theorem to prove that the sufficient decrease condition holds for algorithms in the trust-region framework that we have presented above, let $x_k^M$ be the actual minimum of $m_k$ in the trust region. The result proven in \cite{Trust_region_step} and described by Equation \ref{eqn:tr_opt_sol_result} implies that when $p_k$ is computed according to Algorithm~\ref{alg:tr_sol}, $m_k(x_k) - m(x_k + p_k) \geq (1 - \varepsilon)^2(m_k(x_k) - m_k(x_k^M))$. Since $m_k(x_k^M) \leq m_k(x_k^C)$, Theorem~\ref{thm:C_point} gives that
\begin{equation*}
\begin{aligned}
    &m_k(x_k) - m(x_k + p_k) \geq (1 - \varepsilon)^2(m_k(x_k) - m_k(x_k^C)) \\
    &\geq (1 - \varepsilon)^2 \frac{1}{2} \|g_k\|^2 \min \left\{ \frac{1}{\beta_k},\frac{d_k}{\|g_k\|}\right\},
\end{aligned}
\end{equation*}
implying that the sufficient decrease assumption is fulfilled with $\kappa_{mdc} = \frac{1}{2}(1 - \varepsilon)^2$.

\end{document}